\documentclass{article}

\usepackage{fullpage}
\usepackage{mkolar_definitions}
\usepackage{natbib}
\usepackage{multirow}
\usepackage{rotating}

\usepackage{lmodern}
\usepackage{xcolor}
\usepackage{graphicx}
\usepackage{pstricks-add}

\newbox{\LegendeA}
\savebox{\LegendeA}{
    (\begin{pspicture}(0,0)(0.6,0)
    \psline[linewidth=0.04,linecolor=black](0,0.1)(0.6,0.1)
    \end{pspicture})}
\newbox{\LegendeB}
    \savebox{\LegendeB}{
    (\begin{pspicture}(0,0)(0.6,0)
    \psline[linestyle=dashed,dash=1pt 2pt,linewidth=0.04,linecolor=black](0,0.1)(0.6,0.1)
    \end{pspicture})}

\makeatletter
\let\pen\@undefined
\let\hat\widehat
\newcommand{\pen}{\rho}
\makeatother

\newcommand{\supp}{{\rm supp}}
\newcommand{\uSc}{_{S^{C}}}
\newcommand{\acro}{HIPPO }

\newcommand{\oraclek}{\eqref{eq:known_mean_oracle1}, \eqref{eq:known_mean_oracle2}, \eqref{eq:known_mean_oracle3}}

\begin{document}

\title{\LARGE{Mean and variance estimation in high-dimensional heteroscedastic models with non-convex penalties}}
\author{
\begin{tabular}{ccc}
\Large{James Sharpnack} &\hspace{2cm}& \Large{Mladen Kolar}\\
&\\
Mathematics Department && Booth School of Business\\
UC San Diego && The University of Chicago \\
La Jolla, CA 92093, USA && Chicago, IL 60637, USA\\
\end{tabular}
}

\maketitle

\begin{abstract}
Despite its prevalence in statistical datasets, heteroscedasticity (non-constant sample variances) has been largely ignored in the high-dimensional statistics literature.
Recently, studies have shown that the Lasso can accommodate heteroscedastic errors, with minor algorithmic modifications \citep{Belloni2012Sparse, Gautier2013Pivotal}.
In this work, we study heteroscedastic regression with a linear mean model and a log-linear variances model with sparse high-dimensional parameters.
We propose estimating variances in a post-Lasso fashion, which is followed by weighted-least squares mean estimation.
These steps employ non-convex penalties as in \cite{fan01variable}, which allows us to prove oracle properties for both post-Lasso variance and mean parameter estimates.
We reinforce our theoretical findings with experiments.
\end{abstract}

\section{Introduction}
\label{sec:introduction}

Statistical inference in high-dimensions addresses the problem of
extracting meaningful information from datasets where the number of
variables $p$ can be significantly larger than $n$.  In order to adapt
linear regression to the high-dimensional regime, the statistical and
algorithmic efficiency of penalized least squares methods have been
extensively studied.  Among the most prominent of such procedures is
the Lasso \citep{tibshirani96regression}, Adaptive Lasso
\citep{huang08adaptive}, and the SCAD penalty
\citep{fan01variable}. The majority of this work has focused on mean estimation in
the homoscedastic setting, in which the sample variances are
identical.  In the classical, low-dimensional, setting the effect of
heteroscedasticity and the estimation of variance parameters has been
extensively studied \citep{rutemiller1968estimation,
  carroll1988effect} Recent studies have addressed the problem of mean
estimation under heteroscedasticity in high-dimensions, where the
sample variances may differ \citep{Belloni2012Sparse,
  Gautier2013Pivotal}.  While much of this work has shown that
penalized least squared procedures retain their statistical guarantees
under mild heteroscedasticity, little work has focused on jointly
performing model selection for both mean and variances.  In this work,
we study a simple procedure for estimating both the mean and variance
parameters and examine its ability to correctly identify the sparsity pattern and its asymptotic distribution.

Throughout this work, we assume that we have the usual heteroscedastic Gaussian linear model,
\begin{equation}
  \label{eq:model}
  y_i = \xb_i' \betab^\star + \sigma(\xb_i, \thetab^\star) \epsilon_i, 
  \quad i = 1,\ldots,n,
\end{equation}
where $\xb_i \in \RR^p$ are the observed covariates, 
$\thetab^\star, \betab^\star \in \RR^p$ are unknown parameter,  
$\{\epsilon_i\}_{i=1}^n$ are independent, normally distributed with
mean $0$ and variance $1$ and the function $\sigma$
is of log-linear form, 
\begin{equation}
\label{eq:var_form}
2 \log \sigma(\xb_i, \thetab) = \xb_i' \thetab.
\end{equation}
Modeling the log variance as a linear combination of the explanatory variables, as in \eqref{eq:var_form} is a common choice as it guarantees positivity and is also capable of capturing variance that may vary over several orders of magnitudes \citep{carroll88transformation, harvey76}.
In this paper, we study penalized estimation of the high-dimensional heteroscedastic linear regression model, \eqref{eq:model}, where $p \gg n$.

We study a natural procedure for estimating the mean and variance parameters, called heteroscedastic iterative penalized pseudolikelihood optimizer (HIPPO) first proposed in \citep{Kolar2012Variance}.
This method assumes that a mean estimation procedure such as the lasso is first performed, and then using this mean estimate, one performs model selection for the variance parameter.
Finally, an updated mean parameter is constructed with a regularized weighted least squares procedure.
Thus HIPPO not only outputs a heteroscedasticity aware mean parameter estimate, but also provides variance parameter estimates.
With these parameters the practitioner has an estimate for the predictive distribution given a new sample (by plugging in the estimates $\hat \betab, \hat \thetab$).
Aside from providing superior mean estimates, a primary reason to model variances is that it provides us with estimated predictive distributions.
Furthermore, determining which covariates drive the variance may be of scientific interest.
In economics and finance, volatility of macroeconomic variables and financial instruments is of significant interest. 
(In economic time series heteroscedasticity is modeled in an autoregressive conditional heteroscedastic (ARCH) model \citep{engle1982autoregressive}.)
When rating insurance policies, it is common practice to fit both mean and dispersion parameters in double generalized linear models \citep{peters2009model}, which is a class that the heteroscedastic Gaussian model falls into.
In environmental modeling, climate variability has been recognized as one of the hallmarks of global climate change and has been added to the discussion surrounding the impact of human activity on the environment \citep{karl1995trends}.
More generally, extreme event probabilities are driven primarily by the variance of the predictive distribution for the model, \eqref{eq:model}, so for any application where extreme events are of interest, estimating variances is essential.


Of separate interest is providing confidence regions for mean
parameters in high-dimensions under heteroscedasticity. 
The confidence of an estimate of $\beta_j^\star$ ($j \in \{1,\ldots,p\}$) will be driven by the variances of the samples $\{\sigma(\xb_i,\thetab^\star)\}_{i=1}^n$ in relation to the $j$th covariate $\{\xb_{i,j}\}_{i=1}^n$.
We will see that, given our assumptions, our mean parameter estimates will obtain a specific asymptotic distribution.
This distribution can be inverted to obtain simple confidence region for $\betab^\star$ that is asymptotically valid given our conditions.

The main contributions of this
paper are as follows. First, we review the HIPPO (Heteroscedastic
Iterative Penalized Pseudolikelihood Optimizer) for estimation of both
the mean and variance parameters, and propose some changes to the method. 
Second, we establish theoretical guarantees in the form of oracle properties (in the sense of \citet{fan09nonconcave}) for the estimated
mean and variance parameters.  
These are significantly superior to the theoretical guarantees in \citep{Kolar2012Variance} because they require much more mild assumptions.
We examine some numerical properties of the proposed procedure on a simulation study to complement our theoretical findings.

\subsection{Notation}

Throughout this work matrices and vectors are bolded while scalars are not.
We will let $\Xb = (\xb_1, \ldots, \xb_n)' = (\Xb_1, \ldots, \Xb_p)$ denote the
$n \times p$ matrix of predictors, $\yb, \epsilonb$ are the $n$-vector of responses, and noise respectively.
We will use $O()$ and $o()$ notation to indicate boundedness and convergence of sequences and their probabilistic counterparts $O_\PP(),o_\PP()$.
Throughout the paper we use $[n]$ to denote the set $\{1,\ldots,n\}$.
For any index set $S \subseteq [p]$, we denote $\betab_S$ to be the
subvector containing the components of the vector $\betab$ indexed by
the set $S$, and $\Xb_S$ denotes the submatrix containing the columns
of $\Xb$ indexed by $S$. For a vector $\ab \in \RR^n$, we denote ${\rm
  supp}(\ab) = \{j\ :\ a_j \neq 0\}$ the support set, $\norm{\ab}_q$,
$q \in (0,\infty)$, the $\ell_q$-norm defined as $\norm{\ab}_q =
(\sum_{i\in[n]} a_i^q)^{1/q}$ with the usual extensions for $q \in
\{0,\infty\}$, that is, $\norm{\ab}_0 = |{\rm supp}(\ab)|$ and
$\norm{\ab}_\infty = \max_{i\in[n]}|a_i|$. For notational simplicity,
we denote $\norm{\cdot} = \norm{\cdot}_2$ the $\ell_2$ norm. For a
matrix $\Ab \in \RR^{n \times p}$ we denote $\opnorm{\Ab}{2}$ the
operator norm, $\norm{\Ab}_F$ the Frobenius norm, and
$\Lambda_{\min}(\Ab)$ and $\Lambda_{\max}(\Ab)$ denote the smallest
and largest eigenvalue respectively.

\subsection{Related work}
\textbf{Heteroscedasticity in low-dimensions.}  
Variance parameter estimation in
low-dimensions began with the study of linear forms for the variance, and extended
to higher degree polynomial forms \citep{rutemiller1968estimation,
  geary1966teacher, lancaster1968grouping}.  The estimation of
parameters of the variance when it takes on a log-linear form
in low-dimensions was comprehensively studied in
\cite{harvey76estimating}.  The author concluded that maximum
likelihood, estimated with the iterative `method of scoring', had a
significantly better asymptotic variance than previous methods
proposed.  \cite{carroll1988effect} studied a more general iterative
procedure for variance estimation, specifically they proved limiting
distributions for the mean parameter estimates after a fixed number of
iterations.  

Another classical approach to estimating
variances in the heteroscedastic Gaussian linear model is to use a
restricted likelihood for the estimating equation
\citep{patterson1971recovery}.  
The basic idea is that one can
separate the data $\yb$ into two orthogonal components, one of which is ancillary to $\betab^\star$.
This way the variance parameter $\thetab^\star$ can be estimated from that component by maximizing its marginal likelihood, forming the residual (or restricted) maximum likelihood (REML).
Unfortunately, this method is only valid when $p < n$, so is not suitable for our purposes.

More recently non-parametric procedures were constructed to estimate
variances under heteroscedasticity.  \cite{rigby1996semi} proposed a
general procedure by which a generalized additive model could be
estimated for both the mean and variance (in low dimensions).
\cite{fan98efficient} studied local linear estimates for the mean and
similarly estimating the variance from the resulting residuals.  Other
estimators were constructed from similar procedures under logarithmic
transformations of the residuals \citep{yu2004likelihood,
  chen2009conditional}.  None of these methods are appropriate in
high-dimensions because they do not perform model-selection under a
sparsity assumption.  


Of separate interest is the
effect that heteroscedasticity has on standard regression procedures
that assume homoscedasticity.  While it is the case that ordinary
least squares (OLS) is consistent and enjoys a central limit theorem
despite heteroscedasticity (under mild conditions), the classical
estimate of standard errors is no longer consistent
\citep{eicker1967limit}.  In a landmark paper,
\cite{white1980heteroskedasticity} showed that with minimal
assumptions an estimate of standard errors could be formed by
estimating directly the asymptotic variance of OLS coefficients.
\cite{rao1970estimation} developed an estimator for linear functionals
of the variances which is asymptotically equivalent to that of
\cite{white1980heteroskedasticity}.  This work falls more generally
under the moniker generalized estimating equations (GEE), in which one
presupposes that the likelihood is misspecified and one attempts to
quantify the effect of misspecification on the maximum likelihood
estimates \citep{zieglerGEE, royall1986model}.  (In this way, the
supposed likelihood is called a pseudo-likelihood, which is a name
that we will be using throughout this work.)

\textbf{High dimensional regression.}
The penalization of the empirical loss by the $\ell_1$ norm has become a
popular tool for obtaining sparse models and a vast amount of literature
exists on theoretical properties of estimation procedures
\citep[see,e.g.,][and references therein]{zhao06model,
  wainwright06sharp, zhang09some, zhang08sparsity} and on efficient
algorithms that numerically find estimates \citep[see][for an
extensive literature review]{bach11optimization}.  Due to limitations
of the $\ell_1$ norm penalization, high-dimensional inference methods
based on the class of concave penalties have been proposed that have
better theoretical and numerical properties
\citep[see, e.g., ][]{fan01variable, fan09nonconcave, lv09unified,
  zhang11general}. 

The HIPPO procedure employs a non-convex penalty function to impose sparsity in the estimates.
Nonconvex penalties are commonly used to reduce bias in estimation.  A
number of authors have proposed non-convex penalties, including smoothly clipped absolute deviations (SCAD) \citep{fan01variable}, minimax concave (MC) penalty \citep{zhang2010nearly}, and capped $\ell_1$
\citep{Zhang2010Analysis, Zhang2013Multi}. See \cite{Fan2010selective} for a
recent survey. The oracle estimator is a local solution to the
optimization problem \citep{Kim2008Smoothly, Fan2011Nonconcave}. However, finding
this particular solution is problematic in practice. A number of
recent papers have studied properties of local solutions obtained by
particular numerical procedures (\cite{Zhang2010Analysis},
\cite{Zhang2013Multi}, \cite{Wang2013Calibrating},
\cite{Loh2013Regularized}, \cite{Fan2012Strong},
\cite{Wang2013Optimal}). Under suitable conditions on the design matrix
and the signal size, \cite{Fan2012Strong} and \cite{Loh2013Regularized}
establish that two step procedures obtain the oracle
solution. However, these conditions are quite restrictive. Properties
of global solution to non-convex problem were studied in
\cite{Kim2012Global, Zhang2012General}.

There have been recent advances on providing mean estimates in high dimensions that can handle heteroscedastic errors.
\cite{Belloni2012Sparse} provide an algorithm that adapts the $\ell_1$ penalty to compensate for the effect of different sample variances.
Similarly, \cite{Gautier2013Pivotal}, introduce a family of $\ell_1$ minimization methods called the self-tuned Dantzig estimator which has been shown to handle heteroscedastic errors.
The HIPPO algorithm will use the method of \cite{Belloni2012Sparse} to provide an initial estimate for $\betab^\star$.
There has also been some recent work addressing the estimation of variance parameters in high dimensions.
Notably, \cite{daye2012high} proposes the HHR procedure, that iteratively performs $\ell_1$ penalized likelihood minimizations, but do not provide statistical guarantees.
\cite{cai2008adaptive} developed a wavelet thresholding procedure that is adaptive to the smoothness of the mean and variance functions, but the results are difficult to extend to non-orthogonal design matrices.
\cite{Dalalyan2013Learning} proposes a second-order convex program with group penalties to estimate the mean and variance parameters jointly.
They avoid the likelihoods non-convexity by performing a transformation that makes the likelihood jointly convex, but the choice of transformation (however convenient from an algorithmic standpoint) does not coincide with the log-linear variance model that we consider here.



\section{Methodology}
\label{sec:method}

The primary difficulty with jointly estimating the mean and variance parameters, even in the low dimensional setting (where $p$ is fixed and $n \rightarrow \infty$) is that the likelihood for the model \eqref{eq:model} is not jointly convex in $\betab$ and $\thetab$.
Indeed, the negative log-likelihood for the mean and variance parameters is
\begin{equation}
  \label{eq:loglikelihood}
  \ell(\betab,\thetab; \yb, \Xb) = \sum_{i=1}^n (y_i - \xb_i \betab)^2 \exp (-\xb_i'\thetab) + \xb_i'\thetab,
\end{equation}
up to additive constants.
Because the likelihood with $\thetab$ fixed is convex in $\betab$ and vice versa, a coordinate descent method with a regularized likelihood is a natural approach to joint estimation.
In \citep{Kolar2012Variance}, a simplified method called HIPPO was proposed, where only the first few iterations of the coordinate descent procedure are performed.

The justification for stopping the coordinate descent procedure early is derived from pseudolikelihood theory.
Suppose that we have an initial mean estimate $\hat \betab$, then if we consider minimizing $\ell(\hat \betab, \thetab; \yb, \Xb)$ with $\hat \betab$ fixed, resulting in $\hat \thetab$, then if $\hat \betab$ is close to $\betab^\star$ then $\hat \thetab$ will be close to the minimizer of $\ell(\betab^\star, \thetab; \yb, \Xb)$.
We then think of $\ell(\betab^\star, \thetab; \yb, \Xb)$ as the true likelihood, and $\ell(\hat \betab, \thetab; \yb, \Xb)$ as the pseudo-likelihood.
HIPPO works because the initial lasso procedure produces an estimate for $\hat \betab$ which is good enough in the sense that the pseudo-likelihood is a good approximation of the true likelihood for $\thetab^\star$.
In this paper, we propose some modifications to that algorithm and show that under mild conditions this performs as well as the oracle procedures (where the fixed parameter and the support set of the free parameter is known).
Throughout this work we will consider a penalty function, $\rho_\lambda: \RR_+ \to \RR_+$, which has the effect of enforcing the sparsity in the resulting estimates.

HIPPO is comprised of three steps (which we will refer to as stages):
\begin{enumerate}
\item HIPPO solves a 
  LASSO program (Algorithm A1 in \cite{Belloni2012Sparse}) for
  estimating $\betab^\star$ resulting in $\hat \betab$.

\item HIPPO forms the penalized pseudo-likelihood estimate for $\thetab^\star$ by solving 
\begin{equation}
  \label{eq:stage2}
\begin{aligned}
  \hat\thetab = \arg \min_{\thetab \in \RR^{p}}
  \sum_{i = 1}^n \xb_i'\thetab  + 
  \sum_{i = 1}^n \hat{\eta}_i^2 \exp(-\xb_i'\thetab) 
   + 4 n \sum_{j = 1}^p \rho_{\lambda_j^T}(|\theta_j|)
\end{aligned}
\end{equation}
where $\hat{\etab} = \yb - \Xb \hat{\betab}$ is the vector of residuals.
Furthermore, $\lambda_j^T = \lambda_T \| \Xb_j \|/n$ for $j \in [p]$, where $\lambda_T$ is an appropriately chosen tuning parameter.

\item Finally, HIPPO computes the reweighted estimator of the mean by solving
\begin{equation}
  \label{eq:stage3}
  \hat\betab_w = \arg \min_{\betab \in \RR^{p}}
  \sum_{i = 1}^n \frac{(y_i - \xb_i'\betab)^2}{\hat\sigma^2_i} + 
  2 n \sum_{j = 1}^p \pen_{\lambda_j^S}(|\beta_j|)
\end{equation}
where $\hat\sigma_i = \exp(\xb_i'\hat\thetab/2)$ are the weights.  
Likewise, $\lambda^S_j = \lambda_S n^{-1} \sqrt{\sum_{i = 1}^n x_{i,j}^2 / \hat \sigma_i^2}$ for an appropriately chosen tuning parameter $\lambda_S$.
\end{enumerate}
The details of HIPPO differ here from \citep{Kolar2012Variance} in the
choice of penalty parameters $\lambdab^T, \lambdab^S$ and that stage 1
here uses the LASSO procedure with heteroscedasticity adjusted
penalties of \citet{Belloni2012Sparse}.  These modifications enable us
to demonstrate that HIPPO enjoys significantly stronger theoretical
guarantees than those established in \citet{Kolar2012Variance}.

As was mentioned, the intuition behind HIPPO is that with the
estimate, $\hat \betab$, from stage 1 we can form a penalized
pseudo-likelihood for $\thetab$ given by \eqref{eq:stage2}.  We call
this a pseudo-likelihood, because it can be thought of as an
approximation to the likelihood function for $\thetab^\star$ with
$\betab^\star$ known.  In classical statistics literature, it is
common in misspecified models to consider maximum likelihood methods
as minimizers of an objective function that differs from the true
likelihood.  Central limit theorems have been derived for such maximum
pseudo-likelihood estimators using generalized estimating equations
\citep{zieglerGEE}.  Similarly, in stage 3, the act of reweighting in
effect makes the penalized pseudo-likelihood in \eqref{eq:stage3} much
closer to the true penalized likelihood with known $\thetab^\star$.

HIPPO is closely related to the iterative HHR algorithm of
\citet{daye2012high}.  HIPPO differs for HHR by the choice of penalty
functions and by the fact that we advocate only running the three
stages of HIPPO as opposed to continuing to iterate stages 2 and 3
with the updated $\betab$ and $\thetab$ parameters.  This
recommendation is justified by theoretical and experimental results.

\subsection{Properties of the penalty}

The penalty function, $\pen_\lambda$, is chosen so that the resulting
estimates satisfy three properties: unbiasedness, sparsity and
continuity.  The sparsity condition means that our estimator has the
same support as the true parameter with probability approaching $1$ (a
property commonly known as sparsistency).  Unlike the lasso, our
penalties are chosen so that when the signal size is strong enough the
penalty does not incur a bias on the reconstructed signal.  To provide
us with a theoretical comparison, we can think about the maximum
likelihood estimators that we could construct with the knowledge of
the true sparsity sets, $S = \supp(\betab^\star)$ and $T =
\supp(\thetab^\star)$.  We call these estimators the {\em oracle
  estimators}.  It is our goal to provide minimal conditions under
which HIPPO attains the same asymptotic distribution as the oracle
estimators, and we call this the {\em oracle property}.  The
asymptotic unbiasedness assumption is critical if we hope that HIPPO
will achieve the oracle property.

Concave penalty functions are known to admit solutions that are
asymptotically unbiased.  Examples are the smoothly clipped absolute
deviation (SCAD) penalty \citep{fan01variable}, the minimax concave
(MC) penalty \citep{zhang2010nearly} and a class of folded concave
penalties \citep{lv09unified}.  The SCAD penalty can be defined by its
derivative,
\begin{equation}
\label{eq:SCAD}
\pen_{\lambda}'(\beta) = \lambda \left[ 
     I\{ |\beta| \le \lambda \} +
     \frac{(a \lambda - |\beta|)_+}{(a-1)\lambda} I\{|\beta| > \lambda\} 
   \right],
\end{equation}
where $a > 2$ is a fixed parameter and $\pen_\lambda(0) = 0$.  The
intuition behind the specific form for SCAD is that for a neighborhood
around $0$ it acts like the $\ell_1$ penalty, shrinking small
components toward $0$.  While further from $0$ the effect of the
penalty diminishes until it becomes constant for large enough values
of $\beta$.  Hence, the shrinkage effect is reduced for larger
components, resulting in zero bias in these coordinates.

More generally, the penalty function is assumed to satisfy the
following properties
\begin{enumerate}
\item[\bf (P1)] $\pen_{\lambda}(0) = 0.$
\item[\bf (P2)] The function $\pen_{\lambda}$ satisfies $\pen_{\lambda}(\beta_0+\beta_1)
  \leq \pen_{\lambda}(\beta_0) + \pen_{\lambda}(\beta_1)$ for all $\beta_0,\beta_1 \geq 0$.
\item[\bf (P3)] The derivative $\pen'_{\lambda}(\beta)$ is continuous on $\beta \in (0,
  \infty)$ and normalized so that $\lim_{\beta \rightarrow 0+}
  \pen'_{\lambda}(\beta) = \lambda$.
\item[\bf (P4)] There exists a constant $b > 0$ such that 
  \[
  \begin{aligned}
    \pen_{\lambda}'(\beta) &\leq \lambda, \forall 0 < \beta < b \lambda\\
    \pen_{\lambda}'(\beta) &= 0, \forall \beta \geq b \lambda.
  \end{aligned}
  \]
\end{enumerate}
All of the aforementioned concave penalties satisfy the above
conditions.  For example, the MC penalty \citep{zhang2010nearly} is
likewise defined by its derivative,
\begin{equation}
  \label{eq:mcp_der}
  \pen_{\lambda}'(\beta) = \frac{(a\lambda-\beta)_+}{a},
\end{equation}
where $a>0$ is a fixed parameter and (P1)-(P4) can be verified.

\subsection{Tuning Parameter Selection and Optimization Procedure}

The optimization programs of \eqref{eq:stage2} and \eqref{eq:stage3} require the selection of the
tuning parameters $\lambda_S$ and $\lambda_T$, which balance the propensity to overfit to the data with a complex model and the underfitting when the penalty is too harsh. 
A common approach, and the one that we take, is to form a grid of candidate values for the tuning
parameters $\lambda_S$ and $\lambda_T$ and chose those that minimize
the AIC or BIC criterion
\begin{equation}
  \label{eq:AIC-criterion}
  {\rm AIC}(\lambda_S, \lambda_T) = 
   \ell(\hat\betab, \hat\thetab; \yb, \Xb)
   + 2 \widehat{df},
\end{equation}
\begin{equation}
  \label{eq:BIC-criterion}
  {\rm BIC}(\lambda_S, \lambda_T) = 
   \ell(\hat\betab, \hat\thetab; \yb, \Xb)
   +  \widehat{df} \log n
\end{equation}
where
\[
\widehat{df} = |{\rm supp}(\hat\betab)| + |{\rm supp}(\hat\thetab)|
\]
is the estimated degrees of freedom.  In Section~\ref{sec:simulation},
we evaluate the performance of the AIC and the BIC for HIPPO in
experiments.

While our theoretical results hold for any penalty function, $\pen_\lambda$, with properties (P1)-(P4), in all of our experiments we will use the SCAD penalty defined by \eqref{eq:SCAD}.
We now describe our choice of numerical procedures used to solve the
optimization problems in \eqref{eq:stage2} and
\eqref{eq:stage3}. 
These methods are based on the local linear
approximation for the SCAD penalty developed in \cite{zou08onestep},
\[
\begin{aligned}
\pen_{\lambda}(|\beta_j|) \approx
\pen_{\lambda}(|\beta_j^{(k)}|) + 
\pen_{\lambda}'(|\beta_j^{(k)}|)&(|\beta_j| - |\beta_j^{(k)}|), \quad \text{for }\beta_j \approx \beta_j^{(k)}.
\end{aligned}
\]
With this approximation, we can substitute the SCAD penalty $\sum_{j
  \in [p]}\pen_{\lambda}(|\beta_j|)$ in \eqref{eq:stage2} and \eqref{eq:stage3} with
\begin{equation}  
  \label{eq:scad_sub}
  \sum_{j = 1}^p \pen_{\lambda}'(|\hat\beta_j^{(k)}|)|\beta_j|, 
\end{equation}
and iteratively solve each objective until convergence of
$\{\hat\betab^{(k)}\}_k$. We set the initial estimates
$\hat\betab^{(0)}$ and $\hat\thetab^{(0)}$ to be the solutions of the
$\ell_1$-norm penalized problems.  The convergence of these iterative
approximations follows from the convergence of the MM
(minorize-maximize) algorithms \citep{zou08onestep}.
Recent work has demonstrated that iterative algorithms utilizing local linear expansions of concave penalties have oracle properties (for mean estimation) that hold without the restricted eigenvalue condition \citep{Wang2013Optimal, fan2014strong, Wang2013Calibrating}.

With the approximation of the SCAD penalty given in
\eqref{eq:scad_sub}, we can solve \eqref{eq:stage3} using standard
lasso solvers, for example, we use the proximal method of
\citet{beck09fast}. The objective in \eqref{eq:stage2} is minimized
using a coordinate descent algorithm, which is detailed in
\citet{daye2012high}.

\section{Theoretical Guarantees}

Throughout this section, we will be using the following notation to denote sub-polynomial functions, and polynomial-type decay.
These definitions will be used in the conditions statement and are used throughout the Appendix.
\begin{definition}
We say that a function $f(n)$ is {\em sub-polynomial} if
\[
\forall \gamma > 0, \quad f(n)=o(n^\gamma)
\]
and we denote this by $f(n) = \tilde O(1)$.
We also say that $f(n)$ has {\em polynomial decay} if
\[
\exists \gamma > 0, \quad f(n) = O(n^{-\gamma})
\]
and we denote this with $f(n) = \tilde o(1)$.  
\end{definition}

Because the penalties that we are using are necessarily non-convex, we will not in general have a unique minimum for the program \eqref{eq:stage2}.
As a result, our guarantees will state that there exists a local minimizer that has guarantees similar to what we could achieve had we known the set $T = \supp(\thetab^\star)$ and the parameter $\betab^\star$.
In this sense we say that this local minimizer enjoys {\em oracle properties}.
There is a strong precedent for this style of results in the non-convex penalty literature of \cite{fan01variable,Fan2011Nonconcave}.
Theoretical results of this type suffer from the possibility that the optimization algorithm used may in fact not select this local minimizer, but will instead converge to a local minimum with poor performance. 
These fears will be assuaged by simulation studies.

We will begin our analysis of the second stage program \eqref{eq:stage2}, by considering the variance estimator when $\betab^\star$ is known and we set $\hat \betab = \betab^\star$.
This will serve both as a benchmark and a lemma for the theoretical guarantees of the pseudo-likelihood optimizer in stage 2 (with $\betab^\star$ unknown).

\subsection{Variance estimation with $\betab^\star$ known.}

In the unlikely event that the mean parameter $\betab^\star$ is known, the program \eqref{eq:stage2} may now be considered a true likelihood.
In this setting, we will derive the oracle properties by showing that the oracle maximum likelihood estimate (OMLE), the MLE when the sparsity set $T = \supp(\thetab^\star)$ is known, is a local minimizer of \eqref{eq:stage2}.
As is standard in maximum likelihood theory, we achieve this by examining conditions under which this likelihood is well approximated by its elliptical contours.
The requisite assumptions are listed below.

\begin{enumerate}
\item[\bf (A1)] Define $T = \supp (\thetab^\star)$.
\[
\begin{aligned}
&\opnorm{\Xb_T}{2,\infty} = O \left( \sqrt{t \log p} \right)\\
&\sum_i \| \xb_{i,T} \|^2 = O \left( n t \right)\\
&\sum_i \| \xb_{i,T} \|^3 = O \left( n t^{3/2} \right)\\
\end{aligned}
\]
\item[\bf (A2)] Define the empirical covariance Tensors for $k=2,3,4,6$,
\[
\hat \Sigma = \frac 1n \sum_{i=1}^n \xb_i \xb_i', \quad \hat \Sigma^{(k)} = \frac 1n \sum_{i=1}^n \xb_i^{\otimes k}
\]
Let the following be the largest eigenvalues for the restricted Tensors,
\[
\Lambda_{\max} \left( \hat \Sigma_{TT} \right) = \sup_{\| \zb \| = 1} \frac 1n \sum_{i=1}^n (\xb_{i,T}' \zb)^2, \quad \Lambda_{\max} \left( \hat \Sigma^{(k)}_{T} \right) = \sup_{\| \zb \| = 1} \frac 1n \sum_{i=1}^n (\xb_{i,T}' \zb)^k
\]
then we assume that these are not divergent,
\[
\Lambda_{\max} \left( \hat \Sigma_{TT} \right) = O(1), \quad \Lambda_{\max} \left( \hat \Sigma_{T}^{(k)} \right) = O(1).
\]
And we further assume that the covariance is not singular (asymptotically).
\[
\Lambda_{\min} \left( \hat \Sigma_{TT} \right) = \inf_{\| \zb \| = 1} \frac 1n \sum_{i=1}^n (\xb_{i,T}' \zb)^2 > c_{2}
\]
for some constant $c_{2}$ and $n$ large enough.
\end{enumerate}

We are finally prepared to state the oracle properties of our variance estimator with known mean.
\begin{theorem}
\label{thm:known_mean}
Consider the non-convex program in stage 2, \eqref{eq:stage2}, with $\hat \betab = \betab^\star$ and assume (A1),(A2).
Assume that $t = |T|$ is subpolynomial in $\sqrt n$, $t = \tilde o(\sqrt n)$ and suppose that 
\begin{equation}
\label{eq:incoherence}
\max_{j \in [p]} \left\| \Xb_T \frac{\Xb_j}{\| \Xb_j \|}\right\| = \tilde O\left( n^{-1/4} \right), \quad \max_{i \in [n], j \in [p]} \left| \frac{x_{i,j}}{\| \Xb_j \|} \right| = O((\log p)^{-1/2}),
\end{equation}
that $\log p = \tilde O(1)$, 
\[
\overline \sigma = \max_{i \in [n]} \sigma_i = \tilde O(1), \textrm{ and } \underline \sigma = \min_{i \in [n]} \sigma_i = \tilde \Omega(1).
\] 
We require that the minimal signal size is
\[
\min_{j \in T} |\theta^\star_j| = \omega \left( \frac{\sqrt{\log p}}{\sqrt n} \right).
\]
Then for any sequence, $\lambda_T$, such that
\[
\lambda_T = o \left(\min_{j \in T} |\theta^\star_j| \right), \quad  \frac{\sqrt{\log p}}{\sqrt n} = o(\lambda_T)
\]
there is a local minimizer $\hat \thetab$ such that $\hat \thetab_{T^C} = \zero$ and it enjoys the following {\em oracle properties},
\begin{equation}
\label{eq:known_mean_oracle1}
\textrm{If } \ab \in \RR^p, \textrm{ such that } v = \lim_{n \rightarrow \infty} \ab' \hat \Sigma_{TT}^{-1} \ab \in \RR, \quad \textrm{then } \sqrt n \ab'(\hat \thetab - \thetab^\star) \overset{\Dcal}{\rightarrow} \Ncal(0,v).
\end{equation}
\begin{equation}
\label{eq:known_mean_oracle2}
\max_{j \in [n]} n |\xb_j'(\hat \thetab - \thetab^\star)| = O_\PP \left( \max_{j \in [n]} \left[ \sqrt{n \xb_{j,T}' \hat \Sigma_{TT}^{-1} \xb_{j,T} \log p} + \max_{i \in [n]} \xb_{j,T} \hat \Sigma_{TT}^{-1} \xb_{i,T} \log p\right] \right).
\end{equation}
\begin{equation}
\label{eq:known_mean_oracle3}
\norm{\hat \thetab - \thetab^\star}{} = O_\PP \left( \sqrt{\frac tn} \right).
\end{equation}
\end{theorem}

\begin{remark}
\eqref{eq:known_mean_oracle1} is the distribution achieved by the OMLE because $\ab' \hat \Sigma^{-1}_{TT} \ab$ is the oracle Fisher information for the parameter $\ab'\thetab^\star$.
The simultaneous estimation guarantee of $\xb_j' \thetab^\star$ (and hence of $\sigma(\xb_j,\thetab^\star)$) in \eqref{eq:known_mean_oracle2} is the result of the Bernstein-type inequality for Chi-squared random variables.
\eqref{eq:known_mean_oracle3} demonstrates the rate at which we expect the variance estimate to converge in $\ell_2$ norm.
\end{remark}

Let us begin with a discussion of the conditions in Theorem \ref{thm:known_mean}.
(A1), (A2), and the assumption $t = \tilde o(\sqrt n)$ are conditions required for the OMLE to attain convergence rates akin to those obtained by the central limit theorem in fixed intrinsic dimensions (fixed $T$), hence they would be necessary even in low dimensions.
A note should be made that we could precisely characterize the order of the logarithmic terms in the convergence, $t = \tilde o(\sqrt n)$, and other similar statements, but we choose not to for ease of presentation.
The condition $\log p = \tilde o(1)$ allows for $p = n^k$ and $p = n^{\log^k n}$ for any $k \ge 1$, hence it can accommodate significantly high dimensions.
The condition of \eqref{eq:incoherence} is an artifact of the fact that we are generally dealing with $\chi^2$ random variables in the variance estimation setting, and it seems to be necessary.
The assumption that $\overline \sigma, \underline \sigma^{-1} = \tilde O(1)$ is satisfied by the subGaussian design if $\| \thetab^\star \| = O(1)$, and due to the exponential form for the variance, \eqref{eq:var_form}, in most settings the variance is diverging either subpolynomially or exponentially.

We will refer to \oraclek~collectively as the oracle properties of $\hat \thetab$.  
These results state that under some regularity conditions, there is a local minimizer of \eqref{eq:stage2}, that achieves the low-dimensional rates of convergence.
The minimal signal size that is required has a similar behavior to the rates required by the Lasso for mean estimation under regularity conditions \cite{wainwright06sharp}.
While this result is interesting on its own, Theorem \ref{thm:known_mean} will provide a theoretical benchmark for the unknown mean case.

The difficulty of extending these to the case in which the true mean parameter, $\betab^\star$, is unknown and we only have an estimate is that it is possible that $\hat \betab$ is a function of the variances $\sigma(\xb_i,\thetab^\star)$.
Because we construct $\hat \thetab$ by fitting estimated residuals (constructed by removing the estimated mean from the observations $\hat \eta_i = y_i - \xb_i' \hat \betab$) according to \eqref{eq:stage2}, this correspondence between mean and variance can significantly alter the pseudo-likelihood and its local minima.
In the following section we address these concerns, and show that under mild conditions about the performance of stage 1, there is a local minimizer $\hat \thetab$ constructed with the estimate $\hat \betab$ that attains the same oracle properties as if $\betab^\star$ was known.

\subsection{Variance estimation with $\betab^\star$ unknown.}

We have discussed what is possible when the sample means, $\EE y_i = \xb_i' \betab^\star$, are known. 
Under mild conditions, there is a variance estimate that satisfies the oracle properties, \oraclek, that is also a local minimizer of \eqref{eq:stage2} with $\hat \betab = \betab^\star$.
Of course, in practice, we would have to estimate the mean parameter $\betab^\star$ without knowledge of the unknown variance parameters $\thetab^\star$.
We now show that under reasonable assumptions about the mean estimate $\hat \betab$, and an assumption that the largest sample variance is subpolynomial, we can obtain the oracle properties for a local minimizer of \eqref{eq:stage2} with no additional assumptions on the design.

\begin{theorem}
\label{thm:unknown_mean}
Consider the non-convex program in stage 2, \eqref{eq:stage2}, with $\hat \betab$ satisfying
\begin{equation}
  \label{eq:hat_beta}
    \|\hat \betab\|_0 \le \tilde o_\PP (\sqrt n), \quad \| \Xb (\betab^\star - \hat \betab) \|^2 = \tilde o_\PP(\sqrt n)
\end{equation}
and assume the conditions of Theorem \ref{thm:known_mean} ((A1),(A2), \eqref{eq:incoherence}, $t = \tilde o(\sqrt n)$, and $\overline \sigma, \underline \sigma^{-1}, \log p = \tilde O(1)$).
Similarly, assume that  
\[
\min_{j \in T} |\theta^\star_j| = \omega \left( \frac{\sqrt{\log p}}{\sqrt n} \right).
\]
Then for any sequence, $\lambda_T$, such that
\[
\lambda_T = o \left(\min_{j \in T} |\theta^\star_j| \right), \quad  \frac{\sqrt{\log p}}{\sqrt n} = o(\lambda_T )
\]
there is a local minimizer $\hat \thetab$ such that $\hat \thetab_{T^C} = \zero$ and it enjoys the oracle properties \oraclek.
\end{theorem}

Notice that the conditions placed on $\hat \betab$, \eqref{eq:hat_beta}, are only that the solution is sparse and has a reasonable prediction error.
This mild assumption about the performance of stage 1 allows for there to be a substantial correspondence between the variance and mean parameters ($\thetab^\star$ and $\betab^\star$).
While it may be guessed that the sharing of relevant covariants ($T \cap S \ne \emptyset$ where $S = \supp(\betab^\star)$), not to mention covariate correlations, would be problematic for the pseudo-likelihood minimizer to recover the support of $\thetab^\star$, no such effect is observed.
Only the effect of heteroscedasticity on the ability for $\hat \betab$ to satisfy \eqref{eq:hat_beta} are these concerns manifested.
These results are obtained by considering the fact that we are minimizing a pseudolikelihood for $\thetab$ by plugging in the estimate $\hat \betab$ and showing that the pseudolikelihood is close enough to the likelihood.

In order to ensure that the stage 1 solution satisfies \eqref{eq:hat_beta}, we must impose the restricted eigenvalue condition.
This is a common assumption in the Lasso literature \cite{bickel2009simultaneous}, and is satisfied by subGaussian design \cite{rudelson2011reconstruction}.
We also require that the optimal penalty loadings (to be defined below) in stage 1 are not too large.
\begin{enumerate}
\item[\bf (A3)] Consider the restricted set
\[
\Delta_{C,T} = \{ \delta \in \RR^p : \| \deltab_{T^C} \|_1 \le C \| \deltab_T \|_1, \deltab \ne \zero \}
\]
then the restricted eigenvalue of $\hat \Sigma$ is 
\[
\kappa^2_C(\hat \Sigma) = \min_{\deltab \in \Delta_{C,T}, |T| \le t} \frac{\deltab' \hat \Sigma \deltab}{\| \deltab_T \|_1^2 }
\]
We then assume that the restricted eigenvalue for any $C > 0$ is lower bounded, specifically there exists a constant $\kappa$ such that
\[
\PP \{ \kappa_C(\hat \Sigma) \ge \kappa \} \rightarrow 1
\]
\item[\bf (A4)] The optimal penalty loadings must be not divergent in probability,
\[
\max_{j \in [p]} \frac 1n \sum_{i =1}^n x_{i,j}^2 \sigma_i^2 = O_\PP(1) \textrm{ and }
\min_{j \in [p]} \frac 1n \sum_{i =1}^n x_{i,j}^2 \sigma_i^2 = \Omega_\PP(1)
\]
\end{enumerate}
We now combine this result with results from \cite{Belloni2012Sparse} to show that by using the Lasso solution for stage 1, one obtains the oracle properties in stage 2.

\begin{corollary}
  \label{cor:stage2_belloni}
Consider using algorithm A.1 in \cite{Belloni2012Sparse} as the stage 1 mean estimator.
Assume the conditions of Theorem \ref{thm:known_mean} and that $s = \tilde o(\sqrt n)$, then the stage 2 estimate enjoys the oracle properties \oraclek.
\end{corollary}

\subsection{Mean estimation with weighted least squares in stage 3.}

We have shown that we can obtain accurate variance parameter estimates by solving the penalized pseudolikelihood program in stage 2 under mild conditions on the initial estimate of $\betab^\star$.
With the guarantees of Theorem \ref{thm:unknown_mean}, we will show that the reweighted penalized least squares estimate in stage 3 performs as well as if we had access to the true variances.
To be precise, in the oracle setting, where we have full knowledge of $S = \supp(\betab^\star)$, if in addition we had access to the true variances, then the best linear unbiased estimator (which is also the oracle MLE) would be Gaussian with covariance matrix, $(\Xb_S' \diag(\sigmab^{-2}) \Xb_S)^{-1}$ (the inverse Fisher information).
It is then reasonable to assume that this matrix is well conditioned, if we have any hope of recovering the parameter $\betab^\star$ without prior knowledge of $S$ or $\thetab$.
The following theorem demonstrates that with this mild assumption, under the conditions of Theorem \ref{thm:unknown_mean} and Corollary \ref{cor:stage2_belloni}, we obtain that the stage 3 estimator inherits the asymptotic normality of the oracle MLE just described.
\begin{theorem}
  \label{thm:wls_mean}
Consider the non-convex program in stage 3, \eqref{eq:stage3}, and assume the conditions of Theorem \ref{thm:unknown_mean}.
Denote $\Db = \frac 1n \Xb' \diag(\sigmab^{-2}) \Xb$ and assume that
\[
\Lambda_{\max}(\Db_{SS}) = O(1) \textrm{ and } \Lambda_{\max} (\Db_{SS}^{-1}) = O(1).
\]
Assume that $s = |S|$ is subpolynomial in $\sqrt n$, $s = \tilde o(\sqrt n)$ and suppose that 
\[
\min_{j \in S} |\beta^\star_j| = \omega \left( \frac{\sqrt{\log p}}{\sqrt n} \right).
\]
Then for any sequence, $\lambda_S$, such that
\[
\lambda_S = o \left(\min_{j \in S} |\beta^\star_j| \right), \quad  \frac{\sqrt{\log p}}{n} = o(\lambda_S )
\]
there is a local minimizer $\hat \betab$ such that $\hat \betab_{S^C} = \zero$ and it enjoys the following,
\begin{equation}
\label{eq:wls_oracle1}
\textrm{If } \ab \in \RR^p, \textrm{ such that } v = \lim_{n \rightarrow \infty} \ab' \hat \Db_{SS}^{-1} \ab \in \RR, \quad \textrm{then } \sqrt n \ab'(\hat \betab - \betab^\star) \overset{\Dcal}{\rightarrow} \Ncal(0,v).
\end{equation}
\end{theorem}

Theorem \ref{thm:wls_mean} states that we can achieve the same marginal asymptotic normality property as the oracle MLE.
In fact, in the appendix a stronger statement is proven, specifically that the difference between the oracle MLE and a local minimizer of \eqref{eq:stage3} is of smaller order than the asymptotic variance of oracle MLE.
It should be mentioned that \cite{Belloni2012Sparse} demonstrates that optimal rates can be achieved using the Lasso with appropriately selected penalty.
Theorem \ref{thm:wls_mean} improves on this result by attaining the optimal asymptotic variance for the estimated mean parameter.
This convergence can be inverted to obtain a confidence set, which will be valid under our assumptions.
The significance of Theorem \ref{thm:wls_mean} is that with just the three stages of HIPPO, through the pseudolikelihood approach, we can make a guarantee commensurate with what we would achieve had we known the variances $\sigma_i$.
This complements Theorem \ref{thm:unknown_mean}, and together they provide us with strong guarantees regarding the model selection consistency of both the mean and the variance parameters.

\section{Monte-Carlo Simulations}
\label{sec:simulation}


In this section, we conduct two small scale simulation studies to
demonstrate finite sample performance of \acro.  We compare it to the
HHR procedure \citep{daye2012high} and an oracle procedure that has
additional information.

{\bf Simulation 1.}  In the first scenario, we consider a toy model
where it is assumed that the data are generated from the following
model 
\[
Y = \sigma(\Xb) \epsilon,
\] 
where $\epsilon$ follows a standard normal
distribution and the logarithm of the variance is given by
\[
\log \sigma(\Xb)^2 = X_1 + X_2 + X_3.
\]
The covariates associated with the variance are jointly normal with
equal correlation $\rho$, and marginally $\Ncal(0,1)$. The remaining
covariates, $X_4, \ldots, X_p$ are iid random variables following the
standard Normal distribution and are independent from $(X_1, X_2,
X_3)$. We set $(n, p) = (200, 2000)$ and use $\rho = 0$ and $\rho =
0.5$. For each setting, we average results over 100 independent
simulation runs.

\begin{figure}[p]
\label{fig:sim1:roc}
\begin{center}
\includegraphics[width=0.45 \textwidth]{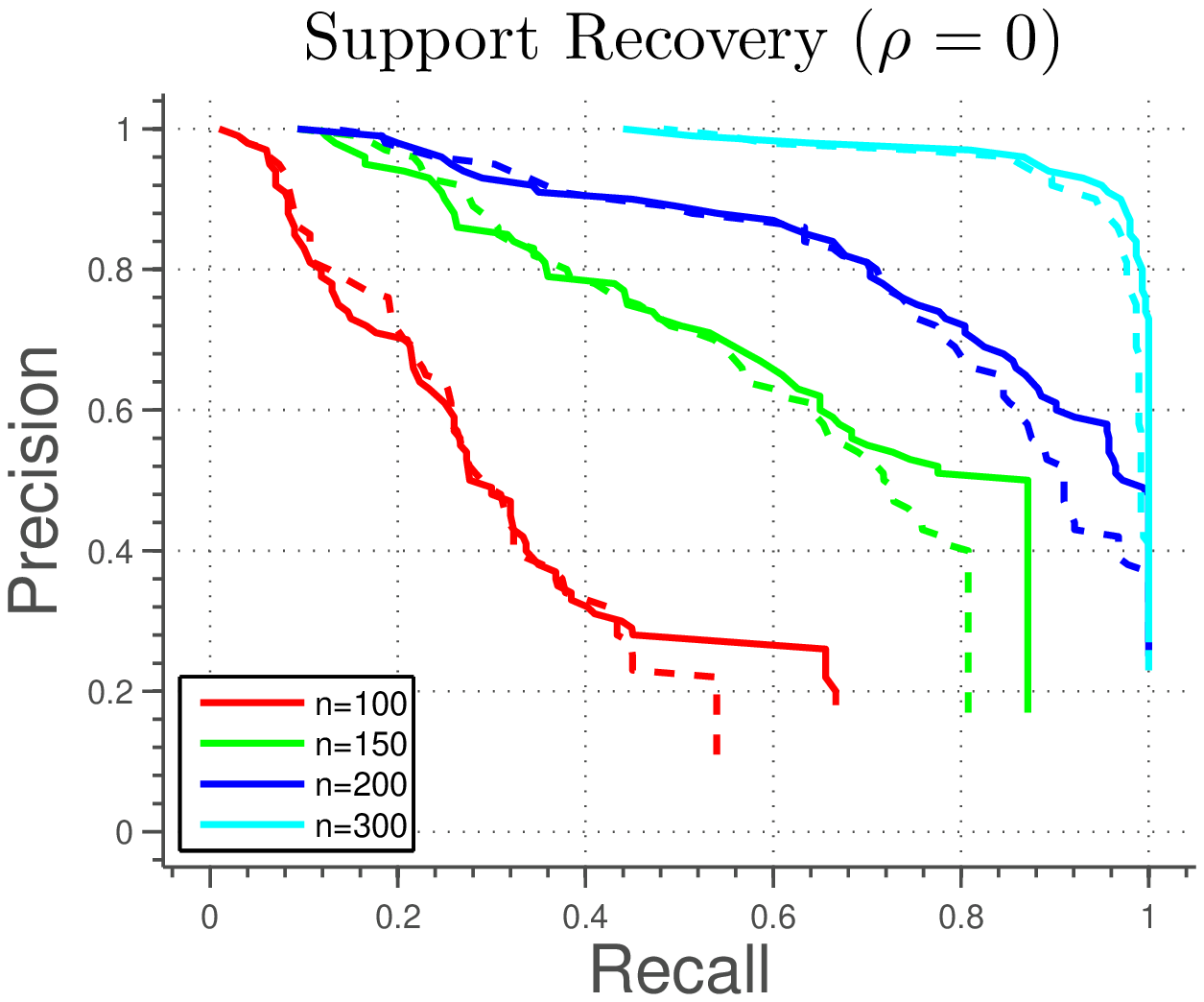}%
\includegraphics[width=0.45 \textwidth]{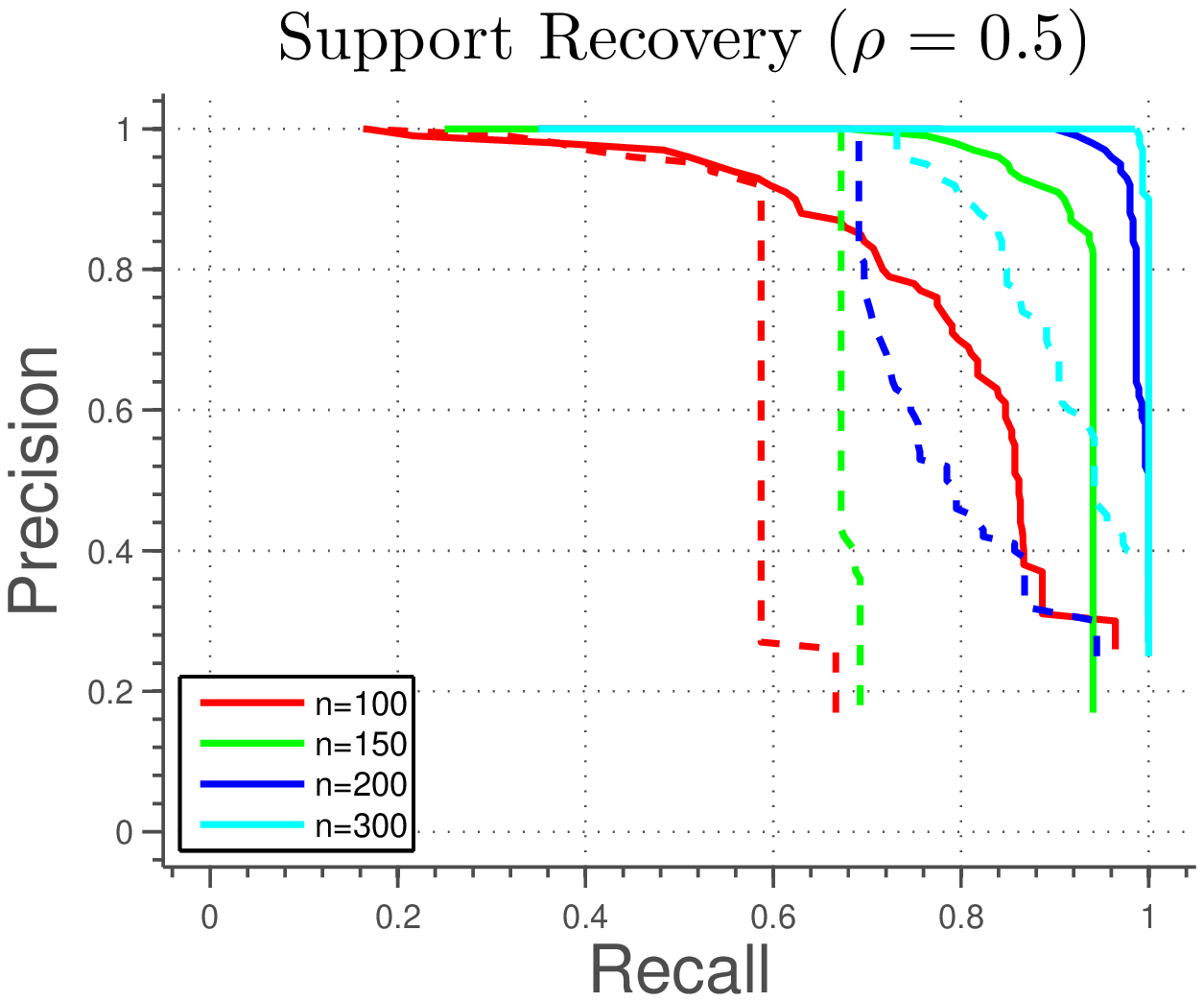}%
\end{center}
\caption{Precision against recall for model in Simulation 1 averaged
  over 100 simulation runs. Full line~\usebox{\LegendeA}  denotes
  results of HIPPO and dashed line~\usebox{\LegendeB}  denotes HHR.}
\end{figure}

\begin{figure}[p]
\label{fig:sim1:l2loss}
\begin{center}
\includegraphics[width=0.45 \textwidth]{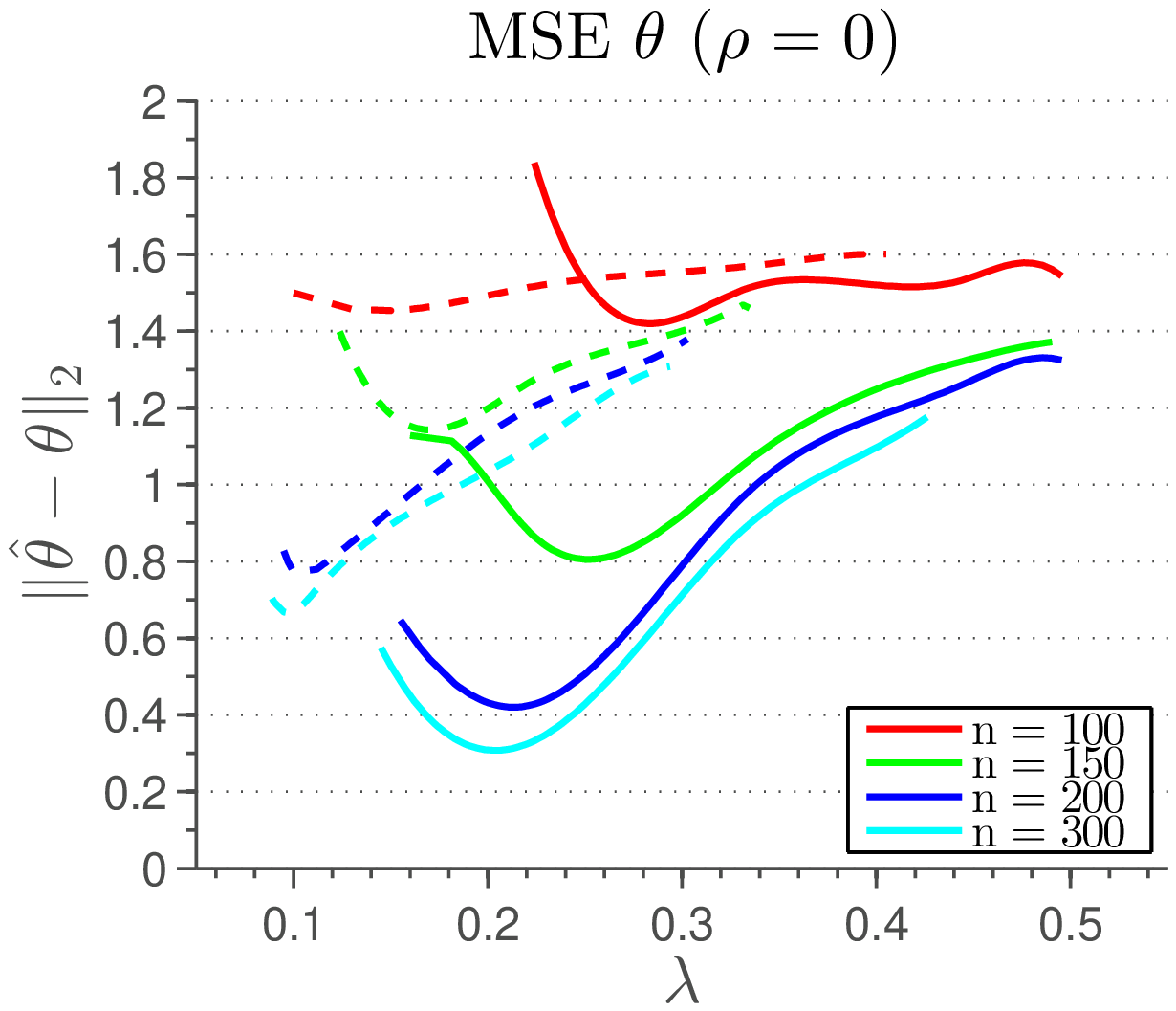}%
\includegraphics[width=0.45 \textwidth]{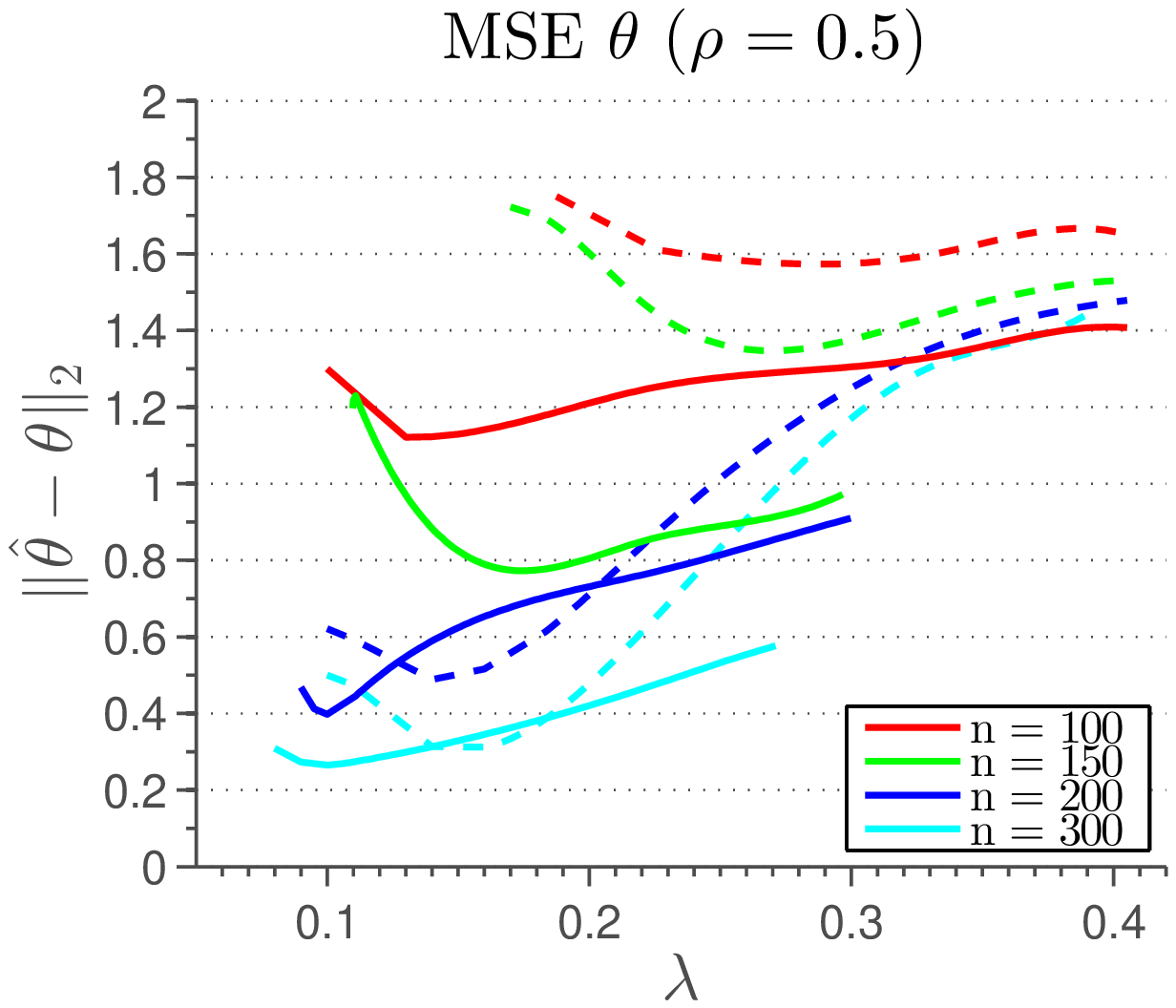}%
\end{center}
\caption{Error in estimating the true parameter $\theta^\star$ as a
  function of the tuning parameter $\lambda^T$
  averaged over 100 simulation runs. Full line~\usebox{\LegendeA}  denotes
  results of HIPPO and dashed line~\usebox{\LegendeB}  denotes HHR.}
\end{figure}

In Simulation 1, it is assumed that the estimation procedures know the
mean parameter, $\betab = \zero$ and we only estimate the variance parameter
$\thetab$. This example is provided to illustrate performance of the
penalized pseudolikelihood estimators in an idealized situation. When
the mean parameter needs to be estimated as well, we expect the
performance of the procedures only to get worse. Since the mean is
known, both HHR and \acro only solve the optimization procedure in
\eqref{eq:stage2}, HHR with the $\ell_1$-norm penalty and \acro with
the SCAD penalty, without iterating between \eqref{eq:stage3} and
\eqref{eq:stage2}. 

Figure~\ref{fig:sim1:roc} shows performance of HIPPO and HHR in
identifying the support of true variance parameter $\theta^\star$
measured by precision and recall\footnote{ We measure the
  identification of the support of $\betab$ and $\thetab$ using
  precision and recall. Let $\hat S$ denote the estimated set of
  non-zero coefficients of $S$, then the precision is calculated as
  ${\sf Pre}_{\beta} := |\hat S \cap S|/|\hat S|$ and the recall as
  ${\sf Rec}_{\beta} := |\hat S \cap S|/|S|$. Similarly, we can define
  precision and recall for the variance coefficients.}.
Figure~\ref{fig:sim1:l2loss} shows $\ell_2$ norm between $\hat \theta$
and $\theta^\star$ as a function of the penalty parameter.  Under this
toy model, we observe that HIPPO performs better than HHR.



\begin{table*}[t]
{
\hfill{}
\begin{tabular}{ll@{\hspace{1cm}}ccc@{\hspace{1cm}}ccc}
& \#it & $\norm{\beta - \hat\beta}_2$  &
     ${\sf Pre}_{\beta}$ & ${\sf Rec}_{\beta}$ 
   & $\norm{\theta - \hat\theta}_2$  &
     ${\sf Pre}_{\theta}$ & ${\sf Rec}_{\theta}$ \\
\cline{3-8}
\vspace{-0.2cm}
\\
&&\multicolumn{6}{c}{\underline{$n = 200$}} \\
\vspace{-0.2cm}
\\

HHR-AIC  & 1st & 0.78(0.52) & 0.44(0.22) & 1.00(0.00) 
           & 2.10(0.11) & 0.25(0.10) & 0.54(0.16) \\
     & 2nd & 0.31(0.13) & 0.88(0.15) & 1.00(0.00) 
           & 1.80(0.16) & 0.29(0.07) & 0.71(0.14) \\
HIPPO-AIC  & 1st & 0.66(0.84) & 0.75(0.29) & 1.00(0.02) 
           & 2.00(0.16) & 0.20(0.10) & 0.52(0.16) \\
     & 2nd & 0.08(0.07) & 0.84(0.24) & 1.00(0.00) 
           & 1.50(0.30) & 0.30(0.11) & 0.75(0.12) \\
\\
HHR-BIC  & 1st & 0.77(0.48) & 0.58(0.17) & 1.00(0.00) 
           & 2.10(0.10) & 0.41(0.18) & 0.45(0.14) \\
     & 2nd & 0.31(0.13) & 0.89(0.13) & 1.00(0.00) 
           & 1.90(0.16) & 0.38(0.15) & 0.65(0.17) \\
HIPPO-BIC  & 1st & 0.70(0.83) & 0.80(0.25) & 0.99(0.03) 
           & 2.00(0.14) & 0.39(0.18) & 0.50(0.17) \\
     & 2nd & 0.08(0.06) & 0.97(0.07) & 1.00(0.00) 
           & 1.60(0.28) & 0.44(0.16) & 0.72(0.14) \\
\\
&&\multicolumn{6}{c}{\underline{$n = 400$}} \\
\vspace{-0.2cm}
\\
HHR-AIC  & 1st & 0.59(0.37) & 0.58(0.26) & 1.00(0.00) 
               & 1.90(0.11) & 0.36(0.14) & 0.72(0.18) \\
     & 2nd & 0.30(0.24) & 0.98(0.06) & 1.00(0.00) 
           & 1.70(0.16) & 0.43(0.13) & 0.81(0.16) \\
HIPPO-AIC  & 1st & 0.44(0.54) & 0.87(0.22) & 1.00(0.00) 
                 & 1.80(0.18) & 0.28(0.10) & 0.67(0.15) \\
     & 2nd & 0.06(0.29) & 0.97(0.12) & 1.00(0.02) 
           & 1.00(0.31) & 0.56(0.18) & 0.93(0.09) \\
\\
HHR-BIC  & 1st & 0.59(0.37) & 0.66(0.20) & 1.00(0.00) 
           & 1.90(0.11) & 0.46(0.18) & 0.66(0.20) \\
     & 2nd & 0.30(0.23) & 0.98(0.06) & 1.00(0.00) 
           & 1.70(0.17) & 0.46(0.13) & 0.80(0.17) \\
HIPPO-BIC  & 1st & 0.46(0.58) & 0.89(0.19) & 1.00(0.01) 
           & 1.80(0.18) & 0.39(0.17) & 0.65(0.17) \\
     & 2nd & 0.06(0.29) & 0.99(0.06) & 1.00(0.02) 
           & 1.00(0.31) & 0.63(0.20) & 0.92(0.09) \\
\hline
\hline
\end{tabular}
}
\hfill{}
\caption{
  Mean (sd) performance of HHR and \acro under the model in
  Example~2 (averaged over 100 independent runs). We report estimated
  models after the first and second iteration. 
}
\label{tb:exper2}
\end{table*}

{\bf Simulation 2.} The following non-trivial model is borrowed from
\citet{daye2012high}.  The response variable $Y$ satisfies
\[
Y = \beta_0 + \sum_{j \in [p]} X_j \beta_j + 
\exp(\theta_0 + \sum_{j \in [p]} X_j\theta_j) \epsilon 
\]
with $p=600$, $\beta_0 = 2$, $\theta_0 = 1$,
\[
\betab_{[12]} = (3, 3, 3, 1.5, 1.5, 1.5, 0, 0, 0, 2, 2, 2)',
\]
\[
\thetab_{[15]} = (1, 1, 1, 0, 0, 0, 0.5, 0.5, 0.5, 0, 0, 0, 0.75, 0.75, 0.75)',
\]
and the remainder of the coefficients are $0$. The covariates are
jointly Normal with ${\rm cov}(X_i, X_j) = 0.5^{|i-j|}$ and the error
$\epsilon$ follows the standard Normal distribution.  We set $p = 600$
and change the sample size.

We first compare performance of HIPPO to an oracle procedure that
knows the mean parameter $\betab^\star$ or the variance parameter
$\theta^\star$. Figure~\ref{fig:sim2:oracle_ROC} shows performance of
HIPPO in recovering the support of $\betab^\star$ and $\thetab^\star$
compared to an oracle
procedure. Figure~\ref{fig:sim2:oracle_l2loss_theta} shows average
$\ell_2$ norm distance between $\hat \thetab$ and $\thetab^\star$.

\begin{figure}[p]
\label{fig:sim2:oracle_ROC}
\begin{center}
\includegraphics[width=0.45 \textwidth]{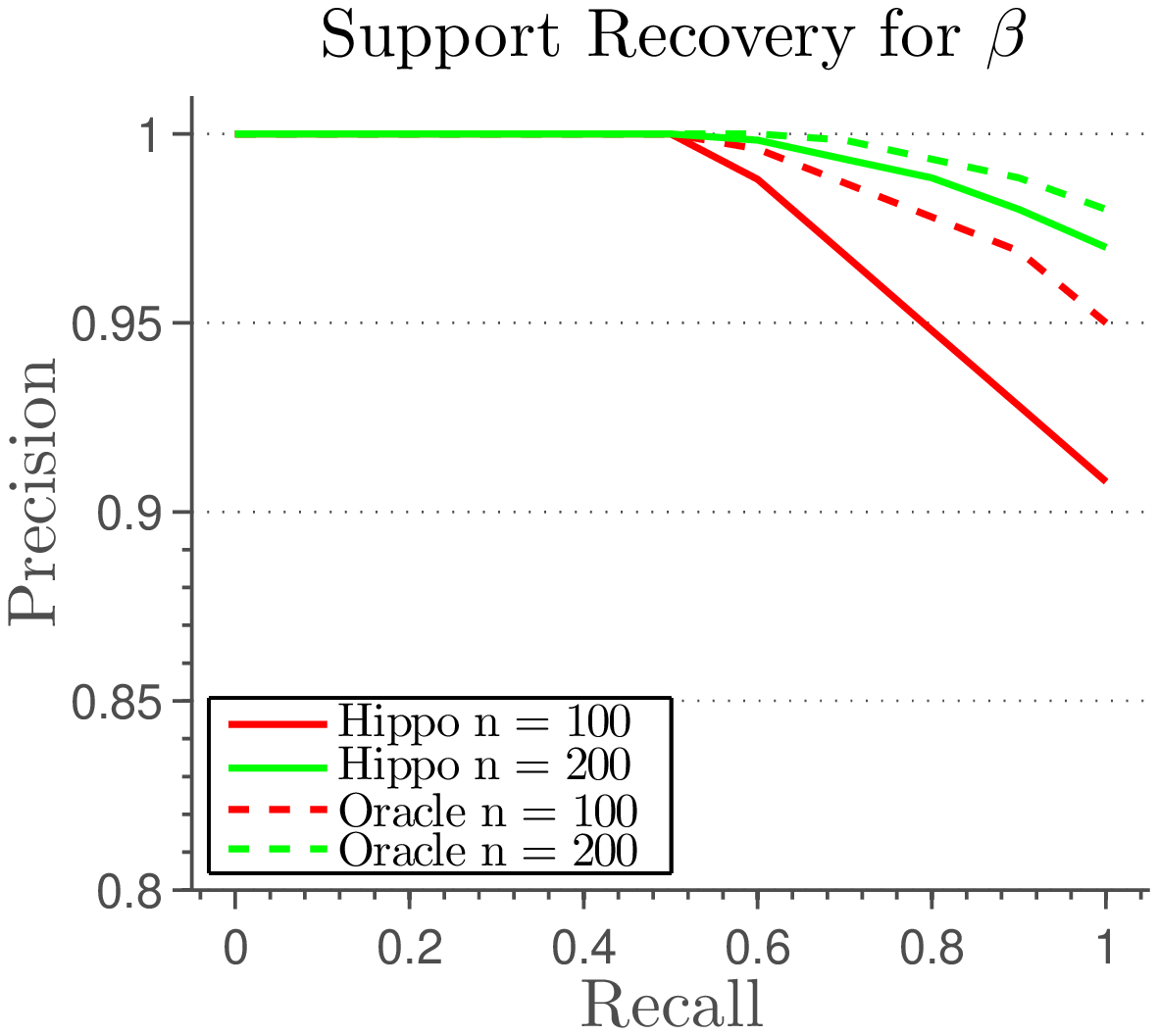}%
\includegraphics[width=0.45 \textwidth]{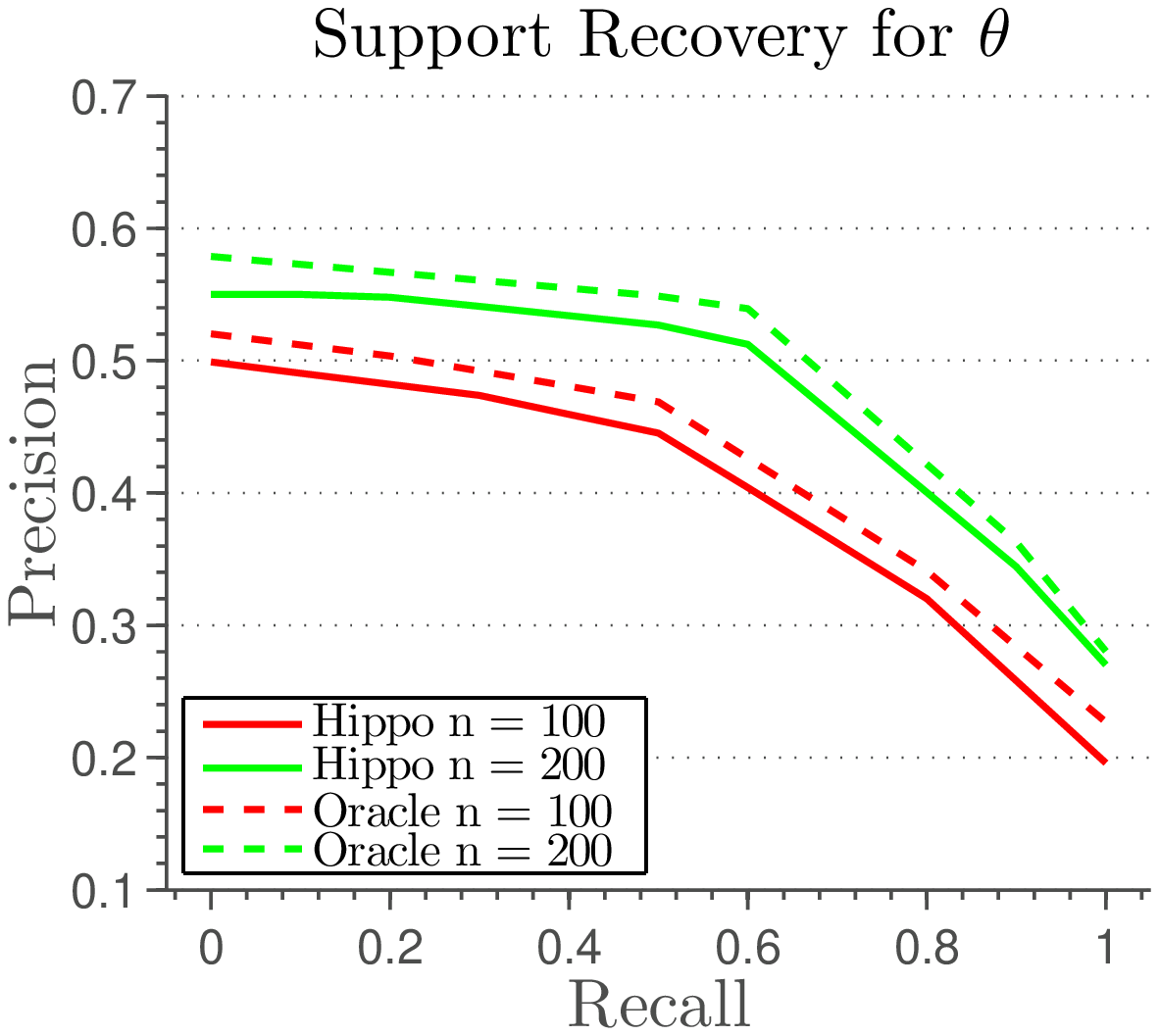}%
\end{center}
\caption{Precision against recall for model in Simulation 2 averaged
  over 100 simulation runs. The oracle procedure is assumed
  to know the true variance (mean) parameter $\thetab^\star$
  $(\betab^\star)$ when estimating the mean (variance).}
\end{figure}

\begin{figure}[p]
\label{fig:sim2:oracle_l2loss_theta}
\begin{center}
\includegraphics[width=0.45 \textwidth]{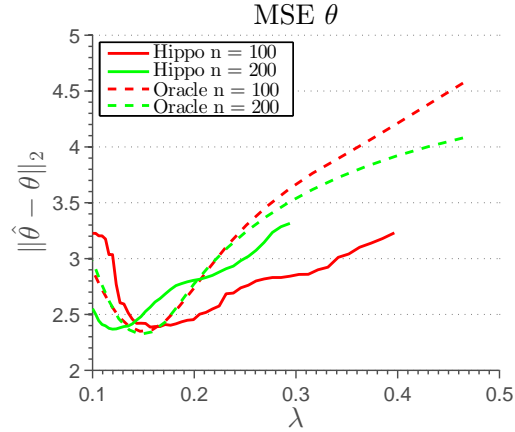}%
\end{center}
\caption{ Error in estimating the true parameter $\theta^\star$ as a
  function of the tuning parameter $\lambda^T$ averaged over 100
  simulation runs. The oracle procedure is assumed to know the true
  mean parameter $\betab^\star$.}
\end{figure}

Next we compare HIPPO to HHR. Table~\ref{tb:exper2} summarizes results
of the simulation. We observe that HIPPO consistently outperforms HHR
in all scenarios. Again, a general observation is that the AIC selects
more complex models although the difference is less pronounced when
the sample size $n=400$. 
Furthermore, we note that the estimation
error significantly reduces after the first iteration, which
demonstrates final sample benefits from estimating the variance. 
While the work of \cite{Belloni2012Sparse} shows that the first stage estimate $\hat \betab$ provides nearly-optimal MSE convergence rates, Theorem~\ref{thm:wls_mean} proves that the third stage can achieve an optimal asymptotic variance.
Hence, it is important to estimate the variance parameter $\thetab^\star$ well, both in
theory and practice.


\section{Discussion}

We have analyzed the performance of HIPPO for estimating mean and variance parameters under heteroscedasticity.
HIPPO is natural because it uses the lasso solution as the first stage, estimates the variances in the second stage, and then adjusts the mean parameters given the variances.
The theoretical statements in Theorems \ref{thm:known_mean}, \ref{thm:unknown_mean} are quite strong because they show that the HIPPO variance estimate, $\hat \thetab$, attains the oracle properties under the same assumptions that are required if the true mean parameter, $\betab^\star$, is known (with mild assumptions on the estimated mean parameter $\hat \betab$).
A similarly strong guarantee is proven for the mean parameter in Theorem \ref{thm:wls_mean}.

Throughout the paper, we assumed that the variance was a log-linear function of its parameters.
One natural extension of this work is to estimate this function in a semi-parametric fashion, such as assuming that the log-variance has a sparse generalized additive form (as in \cite{ravikumar2009sparse}).
HIPPO employs a non-convex penalty (for reasons stated in Section \ref{sec:method}) and it was shown to have favorable performance in Section \ref{sec:simulation}.
Nonetheless, it would be of interest to see what sort of performance guarantees could be made for the lasso penalty.
More generally, the heteroscedastic Gaussian model, \eqref{eq:model}, is a double generalized linear model, and extending this method to other distributions in that family would have applications in insurance and economics.

\subsection*{Acknowledgements}

JS is supported by NSF grant DMS-1223137.
This work was completed in part with resources provided by the University of Chicago Research Computing Center.

\bibliographystyle{my-plainnat}
\bibliography{paper}

\section{Appendix}
\subsection{Technical Lemmata}

\begin{lemma}[\cite{Laurent00adaptive}]
\label{lem:chi_squared}
Let for $i \in \{1,\ldots,p \}$, $a_i \ge 0$ and $\{X_i \}_{i = 1}^p$ be independent $\chi^2_1$ random variables. 
Define $Z = \sum_{i = 1}^p a_i (X_i - 1)$
\[
\begin{aligned}
\PP \{ Z \ge 2 \| \ab \|_2 \sqrt{x} + 2 \| \ab \|_\infty x \} \le e^{-x} \\
\PP \{ Z \le - 2 \| \ab \|_2 \sqrt{x} \} \le e^{-x} 
\end{aligned}
\]
Specifically, this means that 
\[
Z = O_\PP(\| \ab \|_1)
\]
\end{lemma}

\begin{lemma}[\cite{pollard93}]
\label{lem:pollard}
Let $\ell_n(\thetab)$ be a convex function in $t$ dimensions ($t$ possibly growing in $n$).
Consider any quadratic approximation,
\[
\ell_n(\thetab + \tilde \thetab) = \ub' \tilde \thetab + \frac 12 \tilde \thetab' \Vb \tilde \thetab + r(\tilde \thetab)
\]
and let $\tilde \thetab$ denote the argmin.  Let $\Acal \subset \RR^t$ be compact and define the pseudo-norm
\[
\| \xb \|_\Acal = \sup_{\ab \in \Acal} |\ab' \xb|
\]
Let the following be the difference in objectives,
\[
\begin{aligned}
& \Delta(\delta) = \sup \{ |r(\tilde \thetab)| : \|\tilde \thetab - (- \Vb^{-1} \ub) \|_\Acal \le \delta \} 
\end{aligned}
\]
Then
\[
\PP\{ \| \tilde \thetab - (- \Vb^{-1} \ub) \|_\Acal \ge \delta \} \le \PP \{ \Delta(\delta) \ge \frac 12 \underline \lambda \delta^2 \}
\]
where 
\[
\underline \lambda = \min_{\|\xb\|_\Acal = 1} \xb' \Vb \xb.
\]
\end{lemma}

\begin{remark}
This implies that if for any fixed $\delta > 0$, $\Delta(\delta) = o_\PP(1)$ and $\underline \lambda = \Omega(1)$ then
\[
\| \tilde \thetab - (- \Vb^{-1} \ub) \|_\Acal = o_\PP(1)
\]
\end{remark}

\begin{proof}
This proof is based on Lemma 2 in \cite{pollard93}, modified to accommodate the norm $\| .\|_\Acal$.
\end{proof}

\begin{lemma}[\cite{van2000empirical} Lemma 2.5]
\label{lem:VDG_cover}
A ball of radius $R$ in the Euclidean space $\RR^d$ can be covered by 
\[
\left(\frac{4 R + \delta}{\delta}\right)^d
\]
balls of radius $\delta$.
\end{lemma}

\subsection{Outline of Stage 2 Proofs}

The standard procedure for constructing the local minimizer of the least squares objective with a non-convex penalty is to use the maximum likelihood estimator for likelihood with known support, $T$, and demonstrate that this achieves the first-order conditions \cite{fan01variable}.
Situations in which the support is known will be referred to as the oracle setting.
In the known-$\betab^\star$ setting (the setting of Theorem \ref{thm:known_mean}), we will demonstrate first that the oracle MLE where the likelihood is computed using $\hat \betab = \betab^\star$ (we will refer to this estimator as the OMLE for oracle MLE) attains \oraclek.
Using this we will demonstrate that it gives us a local minimizer of \eqref{eq:stage2}, implying that it is the penalized maximum likelihood estimator (we will refer to this as the PMLE), in turn proving Theorem \ref{thm:known_mean}.
We then consider $\hat \betab \ne \betab^\star$ estimated in stage 1, and call the resulting likelihood a pseudo-likelihood.
Similarly to the MLE, we show that the oracle setting for the pseudo-likelihood (we call this estimator the OMPLE) attains \oraclek~under the conditions of Theorem \ref{thm:unknown_mean} using what we have demonstrated regarding the OMLE.
We then show that the OMPLE is in fact a local minimizer of the pseudo-likelihood, so that it is a penalized maximum pseudo-likelihood estimator (PMPLE), in turn proving Theorem \ref{thm:unknown_mean}.
In summary, we show that the OMLE is in fact the PMLE, and then similarly demonstrate that the OMPLE is a PMPLE.

\subsection{Proof of Theorem \ref{thm:known_mean}}

Throughout this section let $\hat \betab = \betab^\star$.
We will begin by proving that the known-$\betab^\star$ MLE, the OMLE, $\hat \thetab_T$ attains the oracle properties.
Then we will show that this is a local minimizer for \eqref{eq:stage2}.

\subsubsection{Oracle property \eqref{eq:known_mean_oracle3} for $\hat \thetab_T$}

Suppose that we know that the true sparsity set $T= \supp(\thetab^\star)$ and we have access to the mean parameter $\betab^\star$.
Thus we can determine precisely,
\[
\eta^2_i = (y_i - \xb_i' \betab^\star)^2 = \epsilon_i^2 e^{\xb_i'\thetab^\star}.
\]
Furthermore, we can minimize the likelihood for $\thetab_T$,
\[
\ell(\thetab_T) = \sum_{i = 1}^n \log \sigma^2_i(\thetab_T) + \frac{(y_i - \xb_i' \betab^\star)^2}{\sigma^2_i(\thetab_T)} = \sum_{i = 1}^n \xb_{i,T}' \thetab_T + \eta_i^2 e^{-\xb_{i,T}'\thetab_T}
\]
The gradient and Hessian of this log-likelihood at the true parameter $\thetab_T^\star$,
\[
\begin{aligned}
\ub(\thetab_T^\star) = \sum_{i = 1}^n (1 - \eta_i^2 e^{-\xb_{i,T}'\thetab_T^\star} ) \xb_{i,T} = \sum_{i = 1}^n (1 - \epsilon_i^2) \xb_{i,T}\\
\Vb(\thetab_T^\star) = \sum_{i = 1}^n \eta_i^2 e^{-\xb_{i,T}'\thetab_T^\star} \xb_{i,T} \xb_{i,T}' = \sum_{i = 1}^n\epsilon_i^2 \xb_{i,T} \xb_{i,T}'.\\
\end{aligned}
\]
Furthermore, the $k$-th derivative tensor of the log-likelihood is
\[
\nabla^{\otimes k} \ell (\thetab_T^\star) = \sum_{i=1}^n \eta_i^2 e^{-\xb_{i,T}'\thetab_T^\star} \xb_{i,T}^{\otimes k}.
\]
For a tensor of the form $A = \sum_{i = 1}^n \ab_i^{\otimes k}$ and a vector $\bb\in \RR^t$ let $A(\bb) = \sum_{i=1}^n (\ab'\bb)^k$.
\begin{lemma}
\label{lem:hess_conc}
Let $k = 2,3$ and
\[
\hat \Sigmab_T^{(k)} = \frac 1n \sum_{i=1}^n \xb_{i,T}^{\otimes k}.
\]
With probability $1 - \delta$, the difference between the $k$-th derivative tensor and $\hat \Sigmab_T^{(k)}$ is bounded by
\[
\Lambda_{\max} \left(\frac 1n \nabla^{\otimes k} \ell (\thetab_T^\star) - \hat \Sigmab_T^{(k)} \right) \le  \frac 2n \left(\sqrt{t n \Lambda_{\max} (\hat \Sigmab_T^{(2k)}) \log(\xi/\delta)} + t \max_{i \in [n]} \|\xb_{i,T}\|^{k} \log (\xi/\delta) \right)
\]
where $\xi$ is some constant only dependent on $k$.
\end{lemma}

\begin{proof}
Let $\ab$ be fixed such that $\| \ab \| = 1$.
\[
\begin{aligned}
&\frac 1n \nabla^{\otimes k} \ell (\thetab_T^\star) - \hat \Sigmab_T^{(k)} = \frac 1n \sum_{i=1}^n (\epsilon_i^2 - 1) \xb_{i,T}^{\otimes k} \\
&\left(\frac 1n \nabla^{\otimes k} \ell (\thetab_T^\star) - \hat \Sigmab_T^{(k)} \right) (\ab^{\otimes k}) = \frac 1n \sum_{i=1}^n (\epsilon_i^2 - 1) (\ab'\xb_{i,T})^k \\
&\le \frac 2n \left(\sqrt{\sum_{i=1}^n (\ab'\xb_{i,T})^{2k} \log(1/\delta)} + \max_{i \in [n]} (\ab'\xb_{i,T})^{k} \log (1/\delta) \right)\\
&\le \frac 2n \left(\sqrt{\sup_{\|\ab\| = 1}\sum_{i=1}^n (\ab'\xb_{i,T})^{2k} \log(1/\delta)} + \max_{i \in [n]} \|\xb_{i,T}\|^{k} \log (1/\delta) \right)
\end{aligned}
\]
by Lemma \ref{lem:chi_squared}.
Let $\Acal \subset \Scal_T$ (where $\Scal_T \subset \RR^T$ is the unit sphere) be a minimal $\xi$-net, meaning that for any $\bb \in \Scal_T$, $\exists \ab \in \Acal$ such that $\| \ab - \bb\| \le \xi$ and $\Acal$ minimizes $|\Acal|$ among all such $\xi$-nets.
Let $\Bb = \frac 1n \nabla^{\otimes k} \ell (\thetab^\star) - \hat \Sigmab_T^{(k)}$.
For some $\bb \in \Scal_T$,
\[
\Lambda_{\max}(\Bb) = \Bb (\bb^{\otimes k}) = \Bb( (\bb' \ab \ab + \sqrt{1 - (\bb'\ab)^2} \bb^{\perp})^{\otimes k} )
\]
where $\ab \in \Acal$ is the closest point to $\bb$ and $\bb^{\perp}$ is a unit vector orthogonal to $\ab$.
Let $\xi >0$ such that
\[
\left( 1 - \frac \xi 2 + \sqrt{1 - (1 - \xi/2)^2} \right)^k - \left( 1 - \frac \xi 2 \right)^k = \frac 12.
\]
By assumption, $\ab'\bb \ge 1 - \xi/2$ ($\|\ab - \bb\| \le \xi$) and
\[
\begin{aligned}
&\Bb( (\bb' \ab \ab + \sqrt{1 - (\bb'\ab)^2} \bb^{\perp})^{\otimes k} ) = \sum_{l = 1}^k {k \choose l} (\ab'\bb)^l (1 - (\ab'\bb)^2)^{\frac{k-l}{2}} \Bb(\ab^{\otimes l} \otimes \bb^{\perp\otimes k-l}) \\
&\le (\ab'\bb)^k \Bb(\ab^{\otimes k}) + \sum_{l = 1}^{k-1} {k \choose l} (\ab'\bb)^l (1 - (\ab'\bb)^2)^{\frac{k-l}{2}} \Lambda_{\max}(\Bb) \le \Bb(\ab^{\otimes k}) + \frac 12 \Lambda_{\max}(\Bb).
\end{aligned}
\] 
Therefore,
\[
\Lambda_{\max}(\Bb) \le 2 \sup_{\ab \in \Acal} \Bb(\ab^{\otimes k}).
\]
Select $\Acal$ to be the covering of the unit ball guaranteed by Lemma \ref{lem:VDG_cover}, which is of size $C^t$ for some constant $C$ (since $\xi$ is a constant).
Hence, we can apply the union bound by substituting $\delta \gets \delta/C^t$, which completes the proof.
\end{proof}

\begin{lemma}
  \label{lem:third_term}
Let $\delta > 0$ be fixed,
\[
\sup_{ \tilde \thetab \in \RR^T, \| \tilde \thetab \| \le \delta} \Lambda_{\max} \left(\frac 1n \nabla^{\otimes 3} \ell \left(\thetab_T^\star + \frac{\sqrt t}{\sqrt n} \tilde \thetab \right) - \frac 1n \nabla^{\otimes 3} \ell (\thetab_T^\star) \right) =  \tilde o_\PP\left( t^{3/2} \right).
\]
\end{lemma}

\begin{proof}
Let $\| \ab \| = 1$,
\[
\frac 1n \left( \nabla^{\otimes 3} \ell (\thetab_T^\star + \frac{\sqrt t}{\sqrt n} \tilde \thetab) - \nabla^{\otimes 3} \ell (\thetab_T^\star) \right) (\ab) = \frac 1n \sum_{i=1}^n \epsilon_i^2 \left( e^{- \frac{\sqrt t}{\sqrt n} \xb_{i,T}'\tilde \thetab} - 1\right) (\xb_{i,T}'\ab)^3 .
\]
Furthermore,
\[
\left| \frac{\sqrt t}{\sqrt n} \xb_{i,T}'\tilde \thetab \right| \le \delta \frac{\sqrt t}{\sqrt n} \max_{i \in [n]} \|  \xb_{i,T} \| = \tilde o(\delta t \sqrt{\log p}/ \sqrt n) = \tilde o (1).
\]
uniformly over $\tilde \thetab$ by assumption (A1).
Thus uniformly over $\| \ab \| = 1$,
\[
\frac 1n \sum_{i=1}^n \epsilon_i^2 \left( e^{-\alpha \frac{\sqrt t}{\sqrt n} \xb_{i,T}'\tilde \thetab} - 1\right) (\xb_{i,T}'\ab)^3 = \tilde o\left((\max_{i \in [n]} \epsilon_i^2) \frac 1n \sum_{i=1}^n \|\xb_{i,T}\|^3 \right) = \tilde o_\PP \left( t^{3/2} \right)
\]
because $\max_{i \in [n]} |\epsilon_i| = \tilde O_\PP(1)$ and by assumption (A1).
\end{proof}

We can Taylor expand the likelihood around $\thetab_T^\star$, (the mean value form where for some $\alpha \in [0,1]$)
\[
\begin{aligned}
&\frac{1}{t} \ell \left( \thetab_T^\star + \sqrt{\frac tn} \tilde \thetab \right) = \frac{1}{\sqrt{tn}} \ub(\thetab_T^\star)' \tilde \thetab + \frac{1}{2n} \tilde \thetab' \Vb(\thetab_T^\star) \tilde \thetab + \frac{\sqrt t}{6n \sqrt{n}} \left( \nabla^{\otimes 3} \ell (\thetab_T^\star + \alpha \sqrt{t/n} \tilde \thetab) \right) (\tilde \thetab) \\
&\frac{1}{t} \ell \left( \thetab_T^\star + \sqrt{\frac tn} \tilde \thetab \right) = \frac{1}{\sqrt{tn}} \ub(\thetab_T^\star)' \tilde \thetab + \frac{1}{2} \tilde \thetab' \hat \Sigma_{TT} \tilde \thetab + r(\tilde \thetab) \\
&\textrm{ where } r(\tilde \thetab) = \frac{1}{2n} \tilde \thetab' \left(\Vb(\thetab_T^\star) - n \hat \Sigma_{TT}\right) \tilde \thetab + \frac{\sqrt t}{6n \sqrt{n}} \left( \nabla^{\otimes 3} \ell (\thetab_T^\star + \alpha \sqrt{t/n} \tilde \thetab) \right) (\tilde \thetab).
\end{aligned}
\]
By showing that the remainder term is uniformly small, we will use the fact that if the likelihood is close to its quadratic approximation (in infinity norm) relative to their curvature of the likelihood then their optima are close.
Fix $\delta > 0$, set $\Acal$ to be the unit ball in $\RR^T$, and define $\Delta(\delta)$ as in Lemma \ref{lem:pollard}.
Then under the (A1),(A2) by Lemma \ref{lem:hess_conc},
\[
\begin{aligned}
&\Lambda_{\max} (\Vb(\thetab_T^\star) - n \hat \Sigma_{TT}) = O_\PP(\sqrt{tn} + t^{3/2}) \\
&\Lambda_{\max} (\nabla^{\otimes 3} \ell (\thetab_T^\star) - n\hat \Sigma^{(3)}_{T}) = O_\PP(\sqrt{n t} + t^2).\\
\end{aligned}
\]
By Lemma \ref{lem:third_term},
\[
\sup_{\|\tilde \thetab\| \le \delta}\Lambda_{\max} (\nabla^{\otimes 3} \ell (\thetab_T^\star + \alpha \sqrt{t/n} \tilde \thetab) - \nabla^{\otimes 3} \ell (\thetab_T^\star )) = \tilde o_\PP(t^{3/2}).
\]
Further by assumption (A2),
\[
\Lambda_{\max} \left( \hat \Sigma_T^{(3)} \right) = O(1).
\]
Combining these observations,
\[
\Delta(\delta) = O_\PP\left(\frac{\sqrt t}{\sqrt n} + \frac{t^{3/2}}{n} + \frac{t}{n} + \frac{t^{5/2} + t^2}{n^{3/2}}\right) = \tilde o_\PP (1), \quad \Lambda_{\min} (\hat \Sigma_{TT} ) = \Omega(1)
\]
because $t = \tilde o(\sqrt n)$. 

Now let us verify that 
\begin{equation}
\label{eq:error_to_grad}
\left\| \frac{1}{\sqrt{tn}} \ub(\thetab_T^\star) \right\| = O_\PP(1).
\end{equation}
Consider 
\[
\left\| \hat \Sigma_{TT}^{-1} \ub(\thetab_T^\star) \right\| = \left\| \sum_{i = 1}^n (1 - \epsilon_i^2) \hat \Sigma^{-1}_{TT} \xb_{i,T} \right\|.
\]
By an identical argument to that used in Lemma \ref{lem:chi_squared},
\[
\begin{aligned}
&\left\| \sum_{i = 1}^n (1 - \epsilon_i^2) \hat \Sigma^{-1}_{TT} \xb_{i,T} \right\| = O_\PP \left(\sqrt{\sum_{i=1}^n \|\hat \Sigma_{TT}^{-1} \xb_{i,T}\|^{2} } + \max_{i \in [n]} \|\hat \Sigma_{TT}^{-1} \xb_{i,T}\| \right) \\  
& = O_\PP \left(\sqrt{\sum_{i=1}^n \| \xb_{i,T}\|^{2} } + \max_{i \in [n]} \| \xb_{i,T}\| \right) = O_\PP (\sqrt{tn})
\end{aligned}
\]
by (A1), (A2), and the fact that $t = \tilde o(\sqrt n)$.

By Lemma \ref{lem:pollard},
\begin{equation}
\label{eq:error_to_grad}
\left\| \sqrt{\frac{n}{t}} (\hat \thetab_T - \thetab_T^\star) - \frac{1}{\sqrt{tn}}(- \hat \Sigma_{TT}^{-1} \ub(\thetab_T^\star)) \right\| = o_\PP(1).
\end{equation}
Further, this implies that 
\[
\left\| \sqrt{\frac{n}{t}} (\hat \thetab_T - \thetab_T^\star) \right\| = O_\PP(1).
\]

\subsubsection{Oracle property \eqref{eq:known_mean_oracle1} for $\hat \thetab_T$.}
We will consider the same expansion of the likelihood as in the previous proof, except that the additional $\sqrt t$ factors are removed.
Let $\tilde \thetab \in \RR^T$.  For some $\alpha \in [0,1]$,
\[
\begin{aligned}
& \ell \left( \thetab_T^\star + \sqrt{\frac 1n} \tilde \thetab \right) = \frac{1}{\sqrt{n}} \ub(\thetab_T^\star)' \tilde \thetab + \frac{1}{2n} \tilde \thetab' \Vb(\thetab_T^\star) \tilde \thetab + \frac{1}{6n \sqrt{n}} \left( \nabla^{\otimes 3} \ell (\thetab_T^\star + \alpha \sqrt{t/n} \tilde \thetab ) \right) (\tilde \thetab) \\
& \ell \left( \thetab_T^\star + \sqrt{\frac 1n} \tilde \thetab \right) = \frac{1}{\sqrt{n}} \ub(\thetab_T^\star)' \tilde \thetab + \frac{1}{2} \tilde \thetab' \hat \Sigma_{TT} \tilde \thetab + r(\tilde \thetab) \\
&\textrm{ where } r(\tilde \thetab) = \frac{1}{2n} \tilde \thetab' \left(\Vb(\thetab_T^\star) - n \hat \Sigma_{TT}\right) \tilde \thetab + \frac{1}{6n \sqrt{n}} \left( \nabla^{\otimes 3} \ell (\thetab_T^\star+ \alpha \sqrt{1/n} \tilde \thetab) \right) (\tilde \thetab).
\end{aligned}
\]
The above proof shows that $r(\tilde \thetab) = \tilde o_\PP(1)$ (since it proves a stronger statement) for $\| \tilde \thetab \| \le \delta$.
Hence, we can employ Lemma \ref{lem:pollard} with $\Acal = \{\ab\}$ to show that,
\[
\left| \sqrt n \ab'(\hat \thetab - \thetab^\star) - \frac{1}{\sqrt n} \left( - \ab' \hat \Sigma^{-1}_{TT} \ub(\thetab^\star) \right) \right| = o_\PP(1)
\]
Because 
\[
\left| \frac{1}{\sqrt n} \ab' \ub(\thetab^\star) \right| = O_\PP(1), \quad \underline \lambda = \ab' \hat \Sigma_{TT} \ab
\]
The rest follows from the CLT.

\subsubsection{Oracle property \eqref{eq:known_mean_oracle2} for $\hat \thetab_T$.}
This is proven similarly to the previous lemmata, except with $\Acal = \{ \xb_{j,T} : j \in [n] \}$.
By Lemma \ref{lem:chi_squared},
\[
\begin{aligned}
&\frac{1}{\sqrt n}\xb_{j,T}' \hat \Sigma_{TT}^{-1} \ub(\thetab_T^\star) = \frac{1}{\sqrt n}\sum_{i = 1}^n (1 - \epsilon_i^2) \xb_{j,T}' \hat \Sigma_{TT}^{-1} \xb_{i,T} \\
&\le 2\left(\sqrt{\frac 1n \sum_{i = 1}^n (\xb_{j,T}'\hat \Sigma_{TT}^{-1} \xb_{i,T})^2 \log (1/\delta) } + \frac{1}{\sqrt n} \max_{i \in [n]} |\xb_{j,T}' \hat \Sigma_{TT}^{-1}\xb_{i,T}| \log (1/\delta) \right) \\
&\le 2\left(\sqrt{\xb_{j,T}'\hat \Sigma_{TT}^{-1} \xb_{j,T} \log (1/\delta) } + \frac{1}{\sqrt n}\max_{i \in [n]} |\xb_{j,T}' \hat \Sigma_{TT}^{-1}\xb_{i,T}| \log (1/\delta) \right).
\end{aligned}
\]
Thus, setting
\[
\begin{aligned}
&v = \max_{j \in [n]} \left| \frac{1}{\sqrt n}\xb_{j,T}' \hat \Sigma_{TT}^{-1} \ub(\thetab_T^\star) \right| = O_\PP \left( \max_{j \in [n]} \sqrt{\xb_{j,T}'\hat \Sigma_{TT}^{-1} \xb_{j,T} \log n } + \frac{1}{\sqrt n}\max_{i \in [n]} |\xb_{j,T}' \hat \Sigma_{TT}^{-1}\xb_{i,T}| \log n \right)\\  
&= O_\PP \left( \sqrt{t \log n} + \frac{t}{\sqrt n} \log n\right) = O_\PP(\sqrt{t \log n}).
\end{aligned}
\]
The above display follows from (A1) and (A2).
Perform the likelihood expansion as in the previous proofs, ($\alpha \in [0,1]$)
\[
\begin{aligned}
&\frac{1}{v^2} \ell \left( \thetab_T^\star + \frac{v}{\sqrt n} \tilde \thetab \right) = \frac{1}{v \sqrt{n}} \ub(\thetab_T^\star)' \tilde \thetab + \frac{1}{2} \tilde \thetab' \hat \Sigma_{TT} \tilde \thetab + r(\tilde \thetab) \\
&\textrm{ where } r(\tilde \thetab) = \frac{1}{2n} \tilde \thetab' \left(\Vb(\thetab_T^\star) - n \hat \Sigma_{TT}\right) \tilde \thetab + \frac{v}{6n \sqrt{n}} \left( \nabla^{\otimes 3} \ell (\thetab_T ^\star+ \alpha \frac{\sqrt t}{\sqrt n} \tilde \thetab) \right) (\tilde \thetab).
\end{aligned}
\]
As in Lemma \ref{lem:pollard}, set $\delta > 0$, and we have that
\[
\Delta(\delta) = O_\PP\left( \frac{1}{2n} \Lambda_{\max}(\Vb(\thetab_T^\star) - n \hat \Sigma_{TT}) + \frac{v}{6n \sqrt{n}} \Lambda_{\max}(\nabla^{\otimes 3} \ell (\thetab_T^\star+\alpha \frac{\sqrt t}{\sqrt n} \tilde \thetab)) \right)
\]
while $v$ is defined so that
\[
\left\| \frac{1}{v \sqrt{n}} \hat \Sigma^{-1}_{TT} \ub(\thetab_T^\star) \right\|_\Acal = O_\PP(1).
\]
By the proof of oracle property \eqref{eq:known_mean_oracle1},
\[
\Delta(\delta) = \tilde o_\PP(1), \quad \Lambda_{\min} (\hat \Sigma_{TT} ) = \Omega(1).
\]
So, if $t = \tilde o(\sqrt n)$, by Lemma \ref{lem:pollard},
\[
\left\| \frac{\sqrt n}{v} (\hat \thetab_T - \thetab_T^\star) - \frac{1}{v \sqrt{n}}(- \hat \Sigma_{TT}^{-1} \ub(\thetab_T^\star)) \right\|_\Acal = o_\PP(1).
\]
Hence,
\[
\max_j \sqrt n |\xb_{j,T}' (\hat \thetab_T - \thetab_T^\star)| = O_\PP(v) = O_\PP \left( \max_{j \in [n]} \sqrt{\xb_{j,T}'\hat \Sigma_{TT}^{-1} \xb_{j,T} \log n } + \frac{1}{\sqrt n}\max_{i \in [n]} |\xb_{j,T}' \hat \Sigma_{TT}^{-1}\xb_{i,T}| \log n \right).
\]

\subsubsection{$\hat \thetab$ is the penalized MLE.}
The estimate $\hat \thetab$ such that $\hat \thetab_T$ is the OMLE (with knowledge of $T$), and $\hat \thetab_{T^C} = \zero$ is a local minimizer of \eqref{eq:stage2} if the following hold (which are precisely the zero-subgradient conditions),
\begin{equation}
\label{eq:var_SCAD_1}
\sum_{i=1}^n (1 - \eta_i^2 e^{-\xb_i' \hat{\thetab}}) x_{i,j} - n \textrm{sgn}(\hat \theta_j) \rho'_{\lambdab^T}(|\hat{\theta}_j|) = 0, \textrm{ if } \theta_j \ne 0
\end{equation}
\begin{equation}
\label{eq:var_SCAD_2}
|\sum_{i=1}^n (1 - \eta_i^2 e^{-\xb_i' \hat{\thetab}}) x_{i,j}| < n \rho_{\lambdab^T}'(0+), \textrm{ if } \theta_j = 0.
\end{equation}

We will focus on \eqref{eq:var_SCAD_2}.
Specifically, we would like to show that for $j \in T^C$,
\[
|\sum_{i=1}^n (1 - \eta_i^2 e^{-\xb_i' \hat{\thetab}}) x_{i,j}| = o_\PP \left( n \lambda^T_j \right)
\]
where recall $\lambda_j^T = \lambda_T \| \Xb_j \| / n$.
We can decompose this term into the following,
\[
\begin{aligned}
\sum_{i=1}^n \left( 1 - \eta_i^2 e^{-\xb_i' \hat{\thetab}} \right) x_{i,j} = \sum_{i=1}^n \left( 1 - \epsilon_i^2 e^{-\xb_i' (\hat{\thetab} - \thetab^\star)} \right) x_{i,j} = \sum_{i=1}^n \left( 1 - e^{-\xb_i' (\hat{\thetab} - \thetab^\star)} \right) x_{i,j} + \left( (1 - \epsilon_i^2) e^{-\xb_i' (\hat{\thetab} - \thetab^\star)} \right) x_{i,j}.
\end{aligned}
\]
We can further decompose the second term,
\[
\sum_{i = 1}^n \left( (1 - \epsilon_i^2) e^{-\xb_i' (\hat{\thetab} - \thetab^\star)} \right) x_{i,j} = \sum_{i = 1}^n (1 - \epsilon_i^2) \left( e^{-\xb_i' (\hat{\thetab} - \thetab^\star)} - 1 \right) x_{i,j} + \sum_{i = 1}^n (1 - \epsilon_i^2) x_{i,j}.
\]
Let us label these terms,
\[
\begin{aligned}
&A_1 = \sum_{i=1}^n \left( 1 - e^{-\xb_i' (\hat{\thetab} - \thetab^\star)} \right) x_{i,j}\\
&A_2 = \sum_{i=1}^n (1 - \epsilon_i^2) \left( e^{-\xb_i' (\hat{\thetab} - \thetab^\star)} - 1 \right) x_{i,j} \\
&A_3 = \sum_{i=1}^n  (1 - \epsilon_i^2) x_{i,j}.
\end{aligned}
\]

Let us begin with the first term ($A_1$).  
By the mean value theorem for each $i$
\[
 \frac 12 (\xb_i' (\hat{\thetab} - \thetab^\star))^2  e^{-|\xb_i' (\hat{\thetab} - \thetab^\star)|} \le 1 - e^{-\xb_i' (\hat{\thetab} - \thetab^\star)} - \xb_i' (\hat{\thetab} - \thetab^\star) \le \frac 12 (\xb_i' (\hat{\thetab} - \thetab^\star))^2 e^{|\xb_i' (\hat{\thetab} - \thetab^\star)|}.
\]
While 
\[
\max_{i \in [n]} |\xb_i' (\hat{\thetab} - \thetab^\star)| \le \max_{i \in [n]} \|\xb_{i,T}\| \delta \le \delta \sqrt{t \log p}
\]
where $\delta = \| \hat \thetab - \thetab^\star \| = O_\PP(\sqrt{t/n})$ by \eqref{eq:known_mean_oracle3}.
Hence, 
\begin{equation}
\label{eq:pred_theta}
\max_{i \in [n]} |\xb_i' (\hat{\thetab} - \thetab^\star)| = o_\PP(1).
\end{equation}
Therefore, uniformly over $j \in T^C$,
\[
\sum_{i=1}^n \left( 1 - e^{-\xb_i' (\hat{\thetab} - \thetab^\star)} \right) x_{i,j} = \sum_{i = 1}^n \xb_i' (\hat{\thetab} - \thetab^\star) x_{i,j} + \left( \frac 12 \sum_{i = 1}^n (\xb_i' (\hat{\thetab} - \thetab^\star))^2 x_{i,j} \right) (1 + o_\PP(1)).
\]

We will first bound the term
\[
\sum_{i = 1}^n \xb_i' (\hat \thetab - \thetab^\star) \frac{x_{i,j}}{ \|\Xb_j\|}.
\]
Define the following
\[
\bb = \frac 1n \sum_{i = 1}^n (\epsilon_i^2 - 1) \hat \Sigma_{TT}^{-1} \xb_{i,T}
\]
then we have that
\[
\| (\hat \thetab - \thetab) - \bb \| = O_\PP(\sqrt{t/n})
\]
by \eqref{eq:error_to_grad}.
Let $\alphab_j = \Xb_j / \| \Xb_j \|$, then by Cauchy-Schwartz
\[
\begin{aligned}
\max_{j \in T^C} \left| \sum_{i = 1}^n \xb_i' (\hat \thetab - \thetab^\star) \alpha_{j,i} - n^{-1} \sum_{i_1, i_2=1}^n \alpha_{j,i_1} (\epsilon_{i_2}^2 - 1) \xb_{i_1, T}' \hat \Sigma_{TT}^{-1} \xb_{i_2,T} \right| \\
\le \max_{j \in T^C} \| (\hat \thetab - \thetab) - \bb \| \| \Xb_T \alphab_j \| = \tilde O_\PP( n^{1/4} \sqrt{t/n}) = \tilde o_\PP(1)
\end{aligned}
\]
by \eqref{eq:incoherence} and the fact that $t = \tilde o(\sqrt n)$.
Define 
\[
\Pb_T = \Xb_T' (\Xb_T \Xb_T')^{-1} \Xb_T
\]
so that 
\[
n^{-1} \sum_{i_1, i_2=1}^n \alpha_{j,i_1} (\epsilon_{i_2}^2 - 1) \xb_{i_1, T}' \hat \Sigma_{TT}^{-1} \xb_{i_2,T} = (\epsilonb^2 - \one)' \Pb_T \alphab_j
\]
and
\[
\max_{j \in T^C} | (\epsilonb^2 - \one)' \Pb_T \alphab_j | = O_\PP \left( \max_{j \in T^C} \| \Pb_T \alphab_j \| \sqrt{\log p} \right) + O_\PP \left( \max_{j\in T^C} \| \Pb_T \alphab_j \|_\infty \log p \right)
\]
by Lemma \ref{lem:chi_squared} combined with the union bound.
Considering the first term, we have
\[
\max_{j \in T^C} \left\| \Pb_T \alphab_j \right\| = n^{-1/2} \Lambda_{\max}(\hat \Sigma_{TT}^{-1/2}) \max_{j \in T^C} \| \Xb_T \alphab_j \|= \tilde O_\PP (n^{-1/4}) = \tilde o_\PP(1)
\]
by \eqref{eq:incoherence} and (A2).
Considering the second term,
\[
\max_{j \in T^C} \| \Pb_T \alphab_j \|_\infty \le \max_{j \in T^C} \| \Pb_T \alphab_j \| = \tilde o(1).
\]
Hence,
\begin{equation}
\label{eq:bt1}
\left| \sum_{i = 1}^n \xb_i' (\hat \thetab - \thetab^\star) \frac{x_{i,j}}{ \|\Xb_j\|} \right| = \tilde o_\PP(1).
\end{equation}
Furthermore, by \eqref{eq:known_mean_oracle1}
\[
\max_{i \in [n]} |\xb_i' (\hat\thetab - \thetab^\star)| = \tilde O_\PP( \sqrt {t/n} )
\]
and so
\[
\sum_{i = 1}^n (\xb_i' (\hat{\thetab} - \thetab^\star))^2 \frac{x_{i,j}}{\| \Xb_j\|} \le \sqrt n \max_{i \in [n]} |\xb_i' (\hat \thetab - \thetab^\star)|^2 = \tilde O_\PP(t / \sqrt n) = o_\PP(1).
\]
In conclusion, $A_1 / \| \Xb_j \| = o_\PP(\sqrt{\log n})$ uniformly in $j\in T^C$.

Focusing on the term $A_3$, by Lemma \ref{lem:chi_squared} and the union bound,
\[
\sup_{j \in T^C} \sum_{i = 1}^n (1 - \epsilon_i^2) \alpha_{i,j} = O_\PP \left( \sqrt{\sum_{i = 1}^n \alpha_{i,j}^2 \log p} + \max_{i,j} |\alpha_{i,j}| \log p \right).
\]
Because $\max_{i,j} \alpha_{i,j} = O(\log^{-1/2} p)$, $A_3 = O_\PP( \sqrt{\log p} )$.

We now focus on $A_2$.
Let $\delta > 0$ depend on $n$, such that $\delta = O(\sqrt{t/n})$ and $\PP \{ \| \hat \thetab - \thetab^\star \| > \delta \} \le \eta_0$ for some fixed $\eta_0 > 0$.
By the mean value theorem,
\[
\max_{i \in [n]}\left| e^{-\xb_i' (\hat{\thetab} - \thetab^\star)} - 1\right|  \le \max_{i \in [n]} |\xb_i'(\hat \thetab - \thetab^\star)| e^{|\xb_i'(\hat \thetab - \thetab^\star)|} = O_\PP (\delta \sqrt{t \log n} )
\]
since by \eqref{eq:pred_theta}, $e^{|\xb_i'(\hat \thetab - \thetab^\star)|} = 1 + o_\PP(1)$.
We will provide a covering argument that controls,
\[
f_j(\thetab) = \sum_{i = 1}^n (1 - \epsilon_i^2) \left( e^{-\xb_i' (\hat{\thetab} - \thetab^\star)} - 1 \right) \frac{x_{i,j}}{\|\Xb_j\|} 
\]
uniformly over $j \in T^C$ and $\Theta = \{ \thetab \in \RR^T : \| \thetab - \thetab^\star \| \le \delta\}$.
Define $\alpha_{i,j} = x_{i,j} / \| \Xb_j\|$.
Consider a pair $\thetab_0, \thetab_1 \in \Theta$,
\[
\begin{aligned}
&\max_{j \in T^C} |f_j(\thetab_0) - f_j(\thetab_1)| \le \max_{j \in T^C} \left| \sum_{i = 1}^n (1 - \epsilon_i^2) \left( e^{-\xb_i' (\thetab_0 - \thetab^\star)} - e^{-\xb_i' (\thetab_1 - \thetab^\star)} \right) \alpha_{i,j} \right| \\
&\le \left( \max_{j \in T^C} \sum_{i = 1}^n |1 - \epsilon_i^2| |\alpha_{i,j}| \right) \max_{i \in [n]} \left| e^{-\xb_i' (\thetab_0 - \thetab^\star)} - e^{-\xb_i' (\thetab_1 - \thetab^\star)} \right|.
\end{aligned}
\]
Now consider balls of radius $\gamma$, and let $a = \max_{i \in [n]} \|\xb_{i,T}\| \delta$, then by the mean value theorem,
\[
\begin{aligned}
&\sup_{\|\thetab_0 - \thetab_1\| \le \gamma} \max_{i \in [n]} \left| e^{-\xb_i' (\thetab_0 - \thetab^\star)} - e^{-\xb_i' (\thetab_1 - \thetab^\star)} \right| 
\le \sup_{\|\thetab_0 - \thetab_1\| \le \gamma} \max_{i \in [n]} e^{|\xb_i' (\thetab_1 - \thetab^\star)|} \left| e^{\xb_i' (\thetab_1 - \thetab_0)} - 1 \right| \\
&= e^a \sup_{\|\thetab_0 - \thetab_1\| \le \gamma} \max_{i \in [n]} | \xb_i' (\thetab_1 - \thetab_0) |e^{|\xb_i' (\thetab_1 - \thetab_0)|} 
\le e^{3a} \gamma \max_{i \in [n]} \| \xb_{i,T} \| = O( \gamma \sqrt{t \log n} )
\end{aligned}
\]
because we have shown that $a = o(1)$.
By Cauchy-Schwartz and the LLN,
\[
\max_{j \in T^C} \sum_{i = 1}^n |1 - \epsilon_i^2| |\alpha_{i,j}| \le \sqrt{ \sum_{i = 1}^n (1 - \epsilon_i^2)^2 } = O_\PP(\sqrt n).
\]
By assumption (A1) if $\gamma = o((\sqrt{t n \log n})^{-1})$, then
\[
\sup_{\|\thetab_0 - \thetab_1\| \le \gamma} \max_{j \in T^C} |f_j(\thetab_0) - f_j(\thetab_1)| = o_\PP(1)
\]
Now we know that we can cover $\Theta$ with an entropy of $O(t \log n)$ by Lemma \ref{lem:VDG_cover}.
For each center in the covering ($\thetab$), with probability $\eta > 0$, by Lemma \ref{lem:chi_squared},
\[
\begin{aligned}
&f_j(\thetab) = \sum_{i = 1}^n (1 - \epsilon_i^2) \left( e^{-\xb_i' (\thetab - \thetab^\star)} - 1 \right) x_{i,j} \\
&\le 2\sqrt{\sum_{i = 1}^n \left( e^{-\xb_i' (\thetab - \thetab^\star)} - 1 \right)^2 \alpha_{i,j}^2 \log (1 / \eta)} + 2\max_{i \in [n]} \left| e^{-\xb_i' (\thetab - \thetab^\star)} - 1 \right| |\alpha_{i,j}| \log (1 / \eta) \\
& \le 4 \max_{i \in [n]} \left| e^{-\xb_i' (\thetab - \thetab^\star)} - 1 \right| \left( \sqrt{\log (1 / \eta)} + \max_{i,j}|\alpha_{i,j}| \log (1/\eta) \right).
\end{aligned}
\]
We have that the convergence in \eqref{eq:pred_theta} is uniform over $\hat \thetab \in \Theta$, so
\[
\sup_{\thetab \in \Theta}|e^{-\xb_i' (\thetab - \thetab^\star)} - 1| = O \left(\max_{i \in [n]}\sup_{\thetab \in \Theta}| \xb_i' (\thetab - \thetab^\star)|\right) = O_\PP( \sqrt{t \log n} / \sqrt n ).
\]
Recall that we have assumed that $\max_{i,j} |\alpha_{i,j}| = \tilde o(n^{1/2} / t^{3/2})$ and $\log p = \tilde O(1)$.
Setting $\eta$ such that $\log(1/\eta) \propto t \log n + \log p$ then
\[
\max_{j \in [p]} \sup_{\thetab} f_j(\thetab) = \tilde O_\PP \left( \sqrt{\frac{t \log n}{n}} \left( \sqrt{t \log n} + \tilde o \left( \frac{\sqrt n }{t \sqrt t} \right) t \log n \right) \right) = \tilde O_\PP\left( \frac{t}{\sqrt n} + \frac{t^2}{n} \right) = \tilde o_\PP(1)
\]
where the supremum in $\thetab$ is over the cover centers.
Because this is higher order than the differences within balls in the cover, we have that 
\[
\max_{j \in [p]} \sup_{\thetab \in \Thetab} f_j(\thetab) = o_\PP(\sqrt{\log p})
\]
because $\eta_0$ (which was a function of $\delta$) can be set to be arbitrarily small.
In conclusion $A_3 / \| \Xb_j \| = o_\PP(\sqrt{\log p})$ uniformly in $j$, which implies \eqref{eq:var_SCAD_2} for $\lambda_T = \omega(\sqrt{\log p})$.

In order to show \eqref{eq:var_SCAD_1}, we can demonstrate that 
\[
\| \hat \thetab_T - \thetab^\star_T \|_\infty = O_\PP\left( \sqrt{\frac{\log p}{n}} \right)
\]
by identical procedures to the proofs of \oraclek~.
Furthermore, for $j \in T$,
\[
\Lambda_{\min} (\hat \Sigma_{TT}) \le n^{-1} \| \Xb_j \|^2 \le \Lambda_{\max} (\hat \Sigma_{TT}).
\]
Hence, if 
\[
\theta^\star_j = \omega\left( \frac{\sqrt{\log p}}{\sqrt n} \right), \quad \forall j \in T
\]
then 
\[
\hat \theta_j = \omega\left( \frac{\sqrt{\log p}}{\sqrt n} \right) = \omega\left( \frac{\sqrt{\log p}}{n} \|\Xb_j\| \right), \quad \forall j \in T
\]
and there is a sequence $\lambda_T$ such that \eqref{eq:var_SCAD_1} holds. 
This demonstrates that $\hat \thetab$ is the PMLE.

\subsection{Proof of Theorem \ref{thm:unknown_mean}}
We now consider $\hat \betab \ne \betab^\star$, but rather it satisfies \eqref{eq:hat_beta}.
We will analyze the oracle pseudo-likelihood setting (the OMPLE).
Throughout this section, we will make the following assumption about the performance of stage 1
\begin{equation}
  \label{eq:A2_4}
  \hat \betab \in \Bcal = \{\betab \in \RR^p : \|\betab\|_0 \le \hat s, \| \Xb (\betab^\star - \betab) \|^2 \le r \}
\end{equation}
for some $\hat s = \tilde o(\sqrt n)$ and $r = O(\hat s)$ as is guaranteed by \eqref{eq:hat_beta}.
The negative log-likelihood and negative log-pseudo-likelihood are given by
\[
\begin{aligned}
\ell(\thetab) = &\sum_{i = 1}^n \log \sigma_i(\thetab)^2 + \frac{(y_i - \xb_i' \betab)^2}{\sigma_i(\thetab)^2} = \sum_{i = 1}^n \log \sigma_i(\thetab)^2 + \epsilon_i^2 e^{\xb_i'(\thetab^\star - \thetab)}\\
\hat \ell(\thetab) = &\sum_{i = 1}^n \log \sigma_i(\thetab)^2 + \frac{(y_i - \xb_i' \hat \betab)^2}{\sigma_i(\thetab)^2} = \sum_{i = 1}^n \log \sigma_i(\thetab)^2 + \epsilon_i^2 e^{\xb_i'(\thetab^\star - \thetab)}\\
& + (\xb_i'(\betab^\star - \hat \betab))^2 e^{-\xb_i'\thetab} + 2 \epsilon_i \xb_i'(\hat \betab - \betab) e^{\frac 12 \xb_i'(\thetab^\star - \thetab)}
\end{aligned}
\]
We will augment the pseudo-likelihood by introducing constants in $\thetab$ which does not affect the minimizer,
\[
\begin{aligned}
&\hat \ell'(\thetab) = \sum_{i = 1}^n \log \sigma_i(\thetab)^2 + \frac{(y_i - \xb_i' \hat \betab)^2}{\sigma_i(\thetab)^2} = \sum_{i = 1}^n \log \sigma_i(\thetab)^2 + \epsilon_i^2 e^{\xb_i'(\thetab^\star - \thetab)}\\
& + (\xb_i'(\betab^\star - \hat \betab))^2 \left(e^{-\xb_i'\thetab} - e^{-\xb_i'\thetab^\star} \right) + 2 \epsilon_i \xb_i'(\hat \betab - \betab ) \left( e^{\frac 12 \xb_i'(\thetab^\star - \thetab)} - 1 \right).
\end{aligned}
\]

\begin{lemma}
\label{lem:MPLE_diff}
Let $C > 0$ and define
\[
\Theta_C = \left\{ \thetab \in \RR^T : \| \Xb (\thetab - \thetab^\star) \|_\infty \le C \sqrt{\frac{t}{n} \log n}, \| \thetab - \thetab^\star \| \le C \sqrt{\frac tn} \right\}
\]
Then the likelihood difference is bounded by
\[
\sup_{\thetab \in \Theta_C} |\hat \ell'(\thetab) - \ell(\thetab)| = o_\PP (t)
\]
\end{lemma}

\begin{proof}
The difference is 
\begin{equation}
\label{eq:l_diff}
\hat \ell'(\thetab) - \ell(\thetab) = \sum_{i = 1}^n (\xb_i'(\betab^\star - \hat \betab))^2 \left(e^{-\xb_i'\thetab} - e^{-\xb_i'\thetab^\star} \right) + 2 \epsilon_i \xb_i'(\hat \betab - \betab ) \left( e^{\frac 12 \xb_i'(\thetab^\star - \thetab)} - 1 \right).
\end{equation}

The first term will be controlled later by a perturbation arguments while the second term requires a covering argument.  
Consider the second term of \eqref{eq:l_diff},
\[
f(\thetab, \hat \betab) = \sum_{i=1}^n \epsilon_i \xb_i'(\hat \betab - \betab^\star) \left(e^{\frac 12 \xb_i'(\thetab^\star - \thetab)} - 1\right)
\]
as a function of $\thetab, \hat \betab$.
We are trying to argue that for any $\thetab \in \Theta$ and $\hat \betab \in \Bcal$, $f(\thetab,\hat \betab)$ is small.
Let $\thetab_0,\thetab_1 \in \Theta$ and $\betab_0, \betab_1 \in \Bcal$.
\[
\begin{aligned}
|f(\thetab_0,\betab_0) - f(\thetab_1,\betab_1)| = \left| \sum_{i=1}^n \epsilon_i \left( \xb_i'(\betab_0 - \betab^\star) e^{\frac 12 \xb_i'(\thetab^\star - \thetab_0)} - \xb_i'(\betab_1 - \betab^\star) e^{\frac 12 \xb_i'(\thetab^\star - \thetab_1)} + \xb_i'(\betab_0 - \betab_1)\right)\right|\\
\le \|\epsilonb\| \sqrt{\sum_{i=1}^n \left( \xb_i'(\betab_0 - \betab^\star) e^{\frac 12 \xb_i'(\thetab^\star - \thetab_0)} - \xb_i'(\betab_1 - \betab^\star) e^{\frac 12 \xb_i'(\thetab^\star - \thetab_1)} \right)^2} + \|\epsilonb\| \| \Xb (\betab_0 - \betab_1) \| \\
\end{aligned}
\]
Consider two pairs, $\deltab_0 = \betab^\star - \betab_0, \deltab_1 = \betab^\star - \betab_1$ and $\thetab_0, \thetab_1 \in \Theta$ and consider the difference in objectives,
\[
\begin{aligned}
&\left| \xb_i'\deltab_0 e^{\frac 12 \xb_i'(\thetab^\star - \thetab_0)} - \xb_i'\deltab_1 e^{\frac 12 \xb_i'(\thetab^\star - \thetab_1)} \right| \\
&\le \left| \xb_i'\deltab_0 \left( e^{\frac 12 \xb_i'(\thetab^\star - \thetab_0)} - e^{\frac 12 \xb_i'(\thetab^\star - \thetab_1)} \right) \right| + \left| \xb_i'(\deltab_0 - \deltab_1) e^{\frac 12 \xb_i'(\thetab^\star - \thetab_1)} \right| \\
&\le \frac 12 |\xb_i' \deltab_0| \| \xb_{i,T} \| \|\thetab_0 - \thetab_1\| e^{\frac 12 \xb_i' (\thetab^\star - \tilde \thetab_i)} + \left| \xb_i'(\deltab_0 - \deltab_1) \right| e^{\frac 12 \xb_i'(\thetab^\star - \thetab_1)}
\end{aligned}
\]
for some $\tilde \thetab_i$ on the segment between $\thetab_1$ and $\thetab_0$ by the mean value theorem.
Assume that $\| \Xb (\deltab_0 - \deltab_1) \| \le \gamma_\beta$.
Assume that $\| \thetab_0 - \thetab_1 \| \le \gamma_\theta$, and define $\overline \sigma^2(\Theta) = \max_{\thetab \in \Theta} \max_i e^{\xb_i'(\thetab^\star - \thetab)}$,
\[
\begin{aligned}
&\sum_{i = 1}^n \left| \xb_i'\deltab_0 e^{\frac 12 \xb_i'(\thetab^\star - \thetab_0)} - \xb_i'\deltab_1 e^{\frac 12 \xb_i'(\thetab^\star - \thetab_1)} \right|^2 \\
&\le \sum_{i=1}^n \frac{\gamma_\theta^2}{2} (\xb_i' \deltab_0)^2 \| \xb_{i,T} \|^2 e^{\xb_i' (\thetab^\star - \tilde \thetab_i)} + 2 \left( \xb_i'(\deltab_0 - \deltab_1) \right)^2 e^{\xb_i'(\thetab^\star - \thetab_1)}\\
&\le \frac{\gamma_\theta^2}{2}  \left(\max_i \| \xb_{i,T} \|^2 \max_{\tilde \thetab_i} e^{\xb_i' (\thetab^\star - \tilde \thetab_i)} \right) \|\Xb \deltab_0\|^2 + 2 \| \Xb (\deltab_0 - \deltab_1)\|^2 \max_i e^{\xb_i'(\thetab^\star - \thetab_1)}\\
&\le r \overline \sigma^2(\Theta) \frac{\gamma_\theta^2}{2} \opnorm{\Xb_T}{2,\infty}^2 + 2 \gamma_\beta^2 \overline \sigma^2(\Theta) = o( rt \gamma_\theta^2 \log n ) + o(\gamma_\beta^2)
\end{aligned}
\]
uniformly over such pairs ($\thetab_0,\thetab_1,\betab_0,\betab_1$) by (A1) and the fact that $\overline \sigma (\Theta) = 1 + o(1)$ (by similar reasoning as in the previous proofs).
Thus, if 
\[
\gamma^2_\theta = O\left( (r t \log n)^{-1} n^{-1/2} \right), \quad \gamma^2_\beta = O(n^{-1/2})
\]
then 
\[
\sup_{\| \thetab_0 - \thetab_1\| \le \gamma_\theta, \|\Xb (\deltab_0 - \deltab_1) \|_2 \le \gamma_\beta}\sum_{i = 1}^n \left| \xb_i'\deltab_0 e^{\frac 12 \xb_i'(\thetab^\star - \thetab_0)} - \xb_i'\deltab_1 e^{\frac 12 \xb_i'(\thetab^\star - \thetab_1)} \right|^2 = o(n^{-1/2})
\]
as well as $\|\Xb(\betab_0 - \betab_1) \| = o(n^{-1/2})$.
These conditions are satisfied by 
\[
\gamma_\theta = n^{-2}, \quad \gamma_\beta = n^{-1}.
\]
We are able to cover the space $\Bcal$ with $e^{O(\hat s \log p)}$ balls of radius $n^{-1}$.
We also can cover the space $\Theta$ with $e^{O(t \log p)}$ of radius $n^{-2}$, hence the metric entropy, $\log N(\gamma)$, is bounded by
\[
\log N(\gamma) = O\left( (\hat s + t) \log p \ \right).
\]
For a fixed $\hat \betab$ and $\thetab$, by Gaussian concentration with probability $1 - \eta$,
\[
\begin{aligned}
&f(\thetab, \hat \betab) = \sum_{i=1}^n \epsilon_i \xb_i'(\hat \betab - \betab^\star) \left( e^{\frac 12 \xb_i'(\thetab^\star - \thetab)} - 1 \right) \\
&\le \sqrt{2 \sum_{i=1}^n (\xb_i'(\betab^\star - \hat \betab))^2 \left( e^{\frac 12 \xb_i'(\thetab^\star - \thetab)} - 1 \right)^2 \log(1/\eta)}.
\end{aligned}
\]
Setting $\eta \propto N(\gamma)^{-1}$ (it can be shown that the first term is dominating based on (A1)),
\[
\sup_{\thetab \in \Theta, \hat \betab \in \Bcal} f(\thetab, \hat \betab) = O_\PP\left( \|\Xb (\betab^\star - \hat \betab) \| \sqrt{\frac{t \log n}{n}} \sqrt{(\hat s+t) \log p} \right) = \tilde o_\PP(t)
\]
by the fact that $r,s,t = \tilde o(\sqrt n)$ and $\log p = \tilde O(1)$.
The first term in the likelihood difference, \eqref{eq:l_diff}, is bounded by
\[
\begin{aligned}
& \| \Xb (\betab^\star - \hat \betab) \|^2 \sup_{\thetab \in \Theta} \max_{i \in [n]} \left|e^{-\xb_i' \thetab} - e^{-\xb_i'\thetab^\star} \right| \le \underline \sigma^{-2} \| \Xb (\betab^\star - \hat \betab) \|^2 \sup_{\thetab \in \Theta} \max_{i \in [n]} \left| e^{\xb_i' (\thetab^\star - \thetab)} - 1 \right|\\
& \le \tilde O_\PP \left( \underline \sigma^{-2} \| \Xb (\betab^\star - \hat \betab) \|^2 \sqrt{\frac{t \log n }{n}} \right) = o_\PP(\sqrt{t}) 
\end{aligned}
\] 
where $\underline \sigma = \min_i \sigma_i$ so that $\underline \sigma^{-2} = \tilde O(1)$.
\end{proof}

\eqref{eq:known_mean_oracle2} and (A1) imply
\[
\begin{aligned}
&\max_{i,j \in [n]} \sqrt{ \xb_{j,T}' \hat \Sigma_{TT}^{-1} \xb_{i,T}} = O \left( \| \Xb_T \|_{2,\infty} \right) = O(\sqrt{t \log n})\\
&\max_{j \in [n]} n |\xb_j'(\hat \thetab - \thetab^\star)| = O_\PP \left( \max_{j \in [n]} \left[ \sqrt{n \xb_{j,T}' \hat \Sigma_{TT}^{-1} \xb_{j,T} \log n} + \max_{i \in [n]} \xb_{j,T} \hat \Sigma_{TT}^{-1} \xb_{i,T} \log n\right] \right) \\
&= O_\PP \left( \sqrt{nt \log n} + t \log n \right) = O_\PP \left( \sqrt{nt} \log n \right)
\end{aligned}
\]
Thus for any $\thetab$ satisfying \oraclek~and for any $\delta > 0$ there exists a $C$ large enough such $\thetab \in \Theta_C$.  

We will now show that the OMPLE is close to the elliptical approximation in the previous subsection.
To distinguish between the OMPLE and the known-$\betab^\star$ MLE, let the OMPLE based on $\hat \betab$ be denoted $\hat \thetab_{\hat \betab}$.
Thus, the known-$\betab^\star$ MLE is denoted by $\hat \thetab_{\betab^\star}$.
Define the norm (for some $\ab$),
\[
\| \thetab \|_r = \left\| \frac{\sqrt n}{\sqrt t} \thetab \right\| + \left\| \frac{\sqrt n}{\sqrt{t \log n}} \Xb \thetab \right\|_\infty + |\ab'\thetab|
\]
Then by Theorem \ref{thm:known_mean}, $\| \hat \thetab_{\betab^\star} - \thetab^\star \|_r = O_\PP \left( 1 \right)$.
By Lemma \ref{lem:MPLE_diff}
\[
\sup_{\| \thetab - \thetab^\star \|_r \le \delta} |\frac 1t \hat \ell'(\thetab) - \frac 1t \ell(\thetab)| = o_\PP(1)
\]
while the curvature of $\ell$ has already been controlled in the previous subsection.
Thus, all of the conditions of Lemma 2 in \cite{pollard93} are satisfied implying,
\[
\| \hat \thetab - \hat \thetab_{\betab^\star} \|_r = o_\PP(1)
\]
This shows all of the oracle properties \oraclek~for the OMPLE.

\subsubsection{Penalized MPLE}

Consider again the pseudolikelihood with $\hat \betab \ne \betab^\star$, we will demonstrate that the OMPLE is a local minimizer of the penalized pseudo-likelihood, and is a PMPLE.
We are now concerned with the gradient,
\[
\sum_{i = 1}^n (1 - \hat \eta_i^2 e^{-\xb_i' \hat \thetab}) \alpha_{i,j} = \sum_{i = 1}^n (1 - \eta_i^2 e^{-\xb_i' \hat \thetab}) \alpha_{i,j} + (\xb_i'(\hat \betab - \betab^\star))^2 e^{-\xb_i \hat \thetab} \alpha_{i,j} + 2 \epsilon_i \xb_i' (\betab^\star - \hat \betab) e^{\frac 12 \xb_i' (\thetab^\star - \hat \thetab)} \alpha_{i,j}
\]
where as before $\alpha_{i,j} = x_{i,j} / \| \Xb_j \|$.

\begin{lemma}
\label{lem:PMPLE_term2}
Let $C>0$, recall the definition of $\Bcal$ in \eqref{eq:A2_4}, and define
\[
\Theta_C = \left\{ \thetab \in \RR^T : \| \thetab - \thetab^\star \|_r \le C \right\}.
\]
Then 
\[
\sup_{\betab \in \Bcal, \thetab \in \Thetab_C} \left| \sum_{i=1}^n (1 - \hat \eta_i^2 e^{-\xb_i'\hat \thetab})\alpha_{i,j} - \sum_{i=1}^n (1 - \eta_i^2 e^{-\xb_i'\hat \thetab})\alpha_{i,j} \right| = \tilde o_\PP(\sqrt n).
\]
\end{lemma}

\begin{proof}
Consider the first term,
\[
\sum_{i = 1}^n (\xb_i'(\hat \betab - \betab^\star))^2 e^{-\xb_i \hat \thetab} \alpha_{i,j} = O_\PP \left( \max_{i \in [n]} \frac{|\alpha_{i,j}|}{\sigma_i^2} \| \Xb (\hat \betab - \betab^\star) \|^2 \right) = \tilde O_\PP(\hat s) = \tilde o_\PP(\sqrt{n}).
\]
For the second term, we will use the familiar covering arguments over the sets $\Bcal, \Theta_C$.
For fixed $\hat \betab, \hat \thetab$,
\[
\begin{aligned}
&\sum_{i = 1}^n \epsilon_i \xb_i' (\betab^\star - \hat \betab) e^{\frac 12 \xb_i' (\thetab^\star - \hat \thetab)} \alpha_{i,j} \le  \sqrt{2 \max_{i\in [n]} \left(e^{\xb_i' (\thetab^\star - \hat \thetab)} |\alpha_{i,j}|^2\right) \| \Xb (\hat \betab - \betab^\star)\|^2 \log(1 / \eta)} \\
&\max_{i \in [n]} e^{ \xb_i' (\thetab^\star - \hat \thetab)} = 1 + o_\PP(1)\\
&\max_{i \in [n], j \in [p]} |\alpha_{i,j}| = O_\PP(1).
\end{aligned}
\]
by \eqref{eq:pred_theta}.
Hence, uniformly over a set of $\hat \betab,\hat\thetab$ of size $N$,
\[
\sum_{i = 1}^n \epsilon_i \xb_i' (\betab^\star - \hat \betab) e^{\frac 12 \xb_i' (\thetab^\star - \hat \thetab)} \alpha_{i,j} = \tilde O_\PP(\sqrt{r \log N}). 
\]
For pairs $\betab_0,\betab_1 \in \Bcal$,
\[
\begin{aligned}
&\left| \sum_{i = 1}^n \epsilon_i \xb_i' (\betab^\star - \betab_0) e^{\frac 12 \xb_i' (\thetab^\star - \hat \thetab)} \alpha_{i,j} - \sum_{i = 1}^n \epsilon_i \xb_i' (\betab^\star - \betab_1) e^{\frac 12 \xb_i' (\thetab^\star - \hat \thetab)} \alpha_{i,j} \right| \\
&= O_\PP \left( \| \epsilon\| \| \Xb (\betab_0 - \betab_1) \| \right) = O_\PP(\sqrt{nr})
\end{aligned}
\]
uniformly over $\hat \thetab \in \Theta_C$.
Recall that, for the set $\Bcal$, the metric entropy (in the norm, $\| \Xb \betab \|$) is $O( \hat s \log p)$, and similarly the metric entropy of the allowed $\hat \thetab$ is $O(t \log p)$.
So the above bound becomes
\[
\max_{j \in T^C} \sum_{i = 1}^n \epsilon_i \xb_i' (\betab^\star - \hat \betab) e^{\frac 12 \xb_i' (\thetab^\star - \hat \thetab)} x_{i,j} = O_\PP(\sqrt{r(\hat s + t ) \log p})
\]
because $\| \Xb (\betab^\star - \hat \betab) \|_\infty \le \| \Xb (\betab^\star - \hat \betab) \|$. 
\end{proof}

Hence, by Lemma \ref{lem:PMPLE_term2},
\[
|\max_{j \in T^C} \sum_{i = 1}^n (1 - \hat \eta_i^2 e^{-\xb_i' \hat \thetab}) x_{i,j}| = \max_{j \in T^C} |\sum_{i = 1}^n (1 - \eta_i^2 e^{-\xb_i' \hat \thetab}) x_{i,j}| + o_\PP(\sqrt n).
\]
So again it is sufficient that 
\[
\lambda_T = \Omega ( \sqrt{\log p} )
\]
for \eqref{eq:var_SCAD_2} to hold by identical reasoning to that of Theorem \ref{thm:known_mean}.
The arguments for \eqref{eq:var_SCAD_1} to hold are identical to that proof.

\subsection{Proof of Corollary \ref{cor:stage2_belloni}.}

We assume (A3) and (A4) throughout this proof.
We must show that the RF condition of \cite{Belloni2012Sparse} holds and that the prescribed restricted eigenvalue constant is $O_\PP(1)$.
The restricted eigenvalue constant is proportional to,
\[
\frac{\max_{j \in [p]} |\Upsilon_j|}{\min_{j \in [p]} |\Upsilon_j|} \textrm{ for }\Upsilon_j = \frac 1n \sum_{i=1}^n x_{i,j}^2 \epsilon_i^2 \sigma_i^2.
\]
It can be shown using $\chi^2$ concentration that 
\[
\max_{j \in [p]} |\Upsilon_j|= O_\PP\left( \frac{1}{n} \sum_{i=1}^n x_{i,j}^2 \sigma_i^2 \right), \quad \min_{j \in [p]} |\Upsilon_j| = \Omega_\PP\left( \frac{1}{n} \sum_{i=1}^n x_{i,j}^2 \sigma_i^2 \right)
\]
 uniformly in $j$, which we assume approaches a constant.

RF (i) follows from the above argument.
RF (ii) follows because $\EE x_{i,j}^3 \eta_i^3 = 0$, due to symmetry.
RF (iii) follows if we further assume that $\log p = \tilde o(n)$, which we have.
By Lemma 3 of \cite{Belloni2012Sparse} RF (iv) holds. 

\section{Proof of Theorem \ref{thm:wls_mean}} 

Let $\hat \Wb = \diag \{ \hat \sigma_i^{-1}\}_{i=1}^n$. 
Consider the WLS estimator with oracle knowledge of $S$ as a function of $\hat \thetab$,
\[
\hat \betab_S(\hat \thetab) = (\Xb_S' \hat \Wb^2 \Xb_S)^{-1} \Xb_S'\hat \Wb^2 \yb.
\]
Let $\hat \Pcal$ be the projection onto the column space of $\hat \Wb \Xb_S$.
Then $\hat \Wb \Xb_S \hat \betab_S(\hat \thetab) = \hat \Pcal \yb$ because it is the WLS estimator and $\hat \Wb \Xb_S'(\yb - \hat \Wb \Xb\hat\betab_S(\hat \thetab)) = \hat \Wb \Xb_S' (\yb - \hat \Pcal \yb) = \zero$.  
Furthermore,
\[
\hat\betab_S(\hat \thetab) - \betab_S^\star = (\Xb_S' \hat \Wb^2 \Xb_S)^{-1} \Xb_S'\hat \Wb^2 \etab = (\Xb_S' \hat \Wb^2 \Xb_S)^{-1} \Xb_S' \hat \Wb^2 {\Wb^\star}^{-1} \epsilonb.
\]
The following is a preliminary that is essential to the remaining proofs.
\begin{lemma}
\label{lem:Dhat}
Consider the reweighted gram matrix as a function of $\thetab$, $\Db(\thetab)_{SS} = \Xb_S' \diag(\sigmab(\thetab)^{-2}) \Xb_S$.
Then for any set $\Theta \subset \RR^p$ such that 
\[
\sup_{\thetab \in \Theta} \max_{i \in [n]} \left|\frac{\sigma_i(\thetab^\star)^2}{\sigma_i(\thetab)^2} - 1 \right| = o_\PP(1)
\]
then 
\[
\sup_{\thetab \in \Theta} \Lambda_{\max} (\Db(\thetab)_{SS}) = \Lambda_{\max} (\Db(\thetab^\star)_{SS}) (1 + o_\PP(1)).
\]
\end{lemma}

\begin{proof}
\[
\begin{aligned}
\Lambda_{\max} (\Db(\thetab)_{SS}) = \sup_{\| \deltab \| = 1} \sum_{i \in [n]} \sigma_i(\thetab)^{-2} (\xb_{i,S}' \deltab)^2 \le \max_{i \in [n]} \left|\frac{\sigma_i(\thetab^\star)^2}{\sigma_i(\thetab)^2} - 1 \right|\sup_{\| \deltab \| = 1} \sum_{i \in [n]} \sigma_i(\thetab^\star)^{-2} (\xb_{i,S}' \deltab)^2.
\end{aligned}
\]
The result easily follows.
\end{proof}
We will now establish that 
\[
\hat \betab_S(\hat \thetab) \approx \hat \betab_S(\thetab^\star)
\]
We state this (and demonstrate what we mean by $\approx$) in the following Lemma.

\begin{lemma}
\label{lem:pseudo_beta_error}
Let $\hat \thetab$ be the second stage estimator and assume the conditions of Theorem \ref{thm:unknown_mean}.
\[
\| \hat \betab_{S}(\hat \thetab) - \hat \betab_{S}(\thetab^\star) \|_\infty = \tilde o_\PP\left( \frac{1}{\sqrt n} \right)
\]
and thus
\[
\| \hat \betab_{S}(\hat \thetab) - \hat \betab_{S}(\thetab^\star) \|_2 = \tilde o_\PP\left( \frac{\sqrt s}{\sqrt n} \right).
\]
\end{lemma}

\begin{proof}
Let us begin with a key lemma.
\begin{lemma}
\label{lem:lipschitz_cover}
Consider $\hat \betab_S(\hat \thetab)$ as a function of $\hat \thetab \in \RR^T$.
Suppose that there is a parameter $L_n$ (possibly growing with $n,p$) such that for each $j \in [p]$, $\hat \betab_j (\thetab)$ is $L_n$-Lipschitz.
Then for any $\xi>0$ there is a covering $\{C_k \subset \RR^T\}_{k=1}^K$ of the unit cube such that for $\thetab_0,\thetab_1 \in C_k$,
\[
\| \hat \betab_S(\thetab_0) - \hat \betab_S(\thetab_1) \|_\infty \le \xi
\] 
the entropy number is bounded by
\[
|K| \le t \log (L_n \sqrt t / \xi).
\]
\end{lemma}

\begin{proof}
Let the cover elements $C_k$ consist of a grid of cubes with side length, $\xi / (L_n \sqrt t)$, it is clear that $|K| = (L_n \sqrt t / \xi)^t$.
Then by the pythagorean theorem the diameter of $C_k$ in $\ell_2$ norm is $\xi / L_n$ and for $\thetab_0, \thetab_1 \in C_k$,
\[
| \hat \betab_j(\thetab_0) - \hat \betab_j(\thetab_1) | \le \xi
\] 
by Lipschitzness.
\end{proof}

Let 
\[
\Lb_j (\hat \thetab) = (\eb_j' (\Xb_S' \hat \Wb \Xb_S)^{-1} \Xb_S' \hat \Wb^2 \Wb^\star)'
\]
so that $\hat \betab_j(\hat \thetab) = \Lb_j(\hat \thetab)' \epsilonb$.
Consider the re-parametrization of $\thetab$ by $\gammab = \Xb \thetab$, so we will write for each $i \in [n]$,
\[
\hat \gammab_i = \xb_i' \hat \thetab, \quad \gammab_i^\star = \xb_i' \thetab^\star.
\]
It is important to notice that the derivative of the variance as a function of $\gammab$ is then
\[
\frac{\partial \sigma^2_i(\gamma_i)}{\partial \gamma_i} = \frac{\partial \exp(\gamma_i)}{\partial \gamma_i} = \sigma^2_i(\gamma_i).
\]
Also, the gradient of $\hat \betab_j$ at $\tilde \thetab$ can be computed by
\[
\frac{\partial \hat \betab_j (\tilde \thetab)}{\partial \tilde \thetab} = \frac{\partial \Lb_j'(\tilde \thetab)}{\partial \tilde \thetab} \epsilonb.
\]
Furthermore, the mean value theorem states that for any $\hat \gammab$ there exists a $\tilde \gammab$ between $\hat \gammab$ and $\gammab^\star$ such that
\[
\hat \betab_S(\hat \gammab) - \hat \betab_S(\gammab^\star) = \frac{\partial \hat \betab (\tilde \gammab)'}{\partial \tilde \gammab} (\hat \gammab - \gammab^\star) = (\hat \gammab - \gammab^\star)' \frac{\partial \Lb_j'(\tilde \gammab)}{\partial \tilde \gammab} \epsilonb  
\]
Further define 
\[
\quad \Gamma = \{ \gammab \in \RR^n: \| \gammab - \gammab^\star\|_\infty \le C \sqrt{\frac tn \log n} \}, \quad \Theta = \{ \thetab \in \RR^T : \Xb \thetab \in \Gamma \}.
\]

\begin{lemma}
\label{lem:beta_theta_grad}
Define $\tilde \Db = \Db(\tilde \thetab)$ for $\tilde \thetab \in \Theta$.
Assume that
\[
\Lambda_{\max}\left(\tilde \Db_{SS}^{-1} \right) = \tilde O_\PP(1), \quad \Lambda_{\max}\left(\tilde \Db_{SS} \right) = \tilde O_\PP(1)
\]
uniformly over $\tilde \thetab \in \Theta$.
The following term appears in the use of the mean-value theorem and is bounded by,
\[
\sup_{\tilde \gammab \in \Gamma} \left\| \frac{\partial \Lb_j(\tilde \gammab)}{\partial \tilde \gammab} (\hat \gammab - \gammab^\star) \right\| = \tilde O_\PP \left( \frac{\sqrt{t}}{n} \right).
\]
There exists a constant $q > 0$ such that 
\[
\sup_{\tilde \thetab} \opnorm{ \frac{\partial \Lb_j(\tilde \thetab)}{\partial \tilde \thetab} }{} = O_\PP(n^q).
\]
\end{lemma}

\begin{proof}
Evaluating the partial derivative, (and let $\tilde \Db, \tilde \Wb, \tilde \sigmab$ be the Gram matrix, weight matrix, and standard deviations defined using $\tilde \gammab$)
\[
\begin{aligned}
&n \frac{\partial \Lb'_j(\tilde \gamma)}{\partial \tilde \gamma_k} = \eb_j' \left( \frac{\partial}{\partial \tilde \gamma_k} (\tilde \Db_{SS})^{-1} \right) \Xb_S' \tilde \Wb^2 (\Wb^{\star})^{-1} + \eb_j' (\tilde \Db_{SS})^{-1} \left( \frac{\partial}{\partial \gamma_k} \Xb_S' \tilde \Wb^2 (\Wb^{\star})^{-1} \right) \\
& = - \eb_j' \tilde \Db_{SS}^{-1} \left( \frac{\partial}{\partial \gamma_k} \tilde \Db_{SS} \right) \tilde \Db_{SS}^{-1} \Xb_S' \tilde \Wb^2 (\Wb^{\star})^{-1}  +  \eb_j' (\tilde \Db_{SS})^{-1} \left( \xb_{k,S} \frac{\tilde \sigma_k^2}{\sigma^\star_k} \eb_k' \right) \\
& = - \frac 1n \eb_j' \tilde \Db_{SS}^{-1} \left( \tilde \sigma_k^2 \xb_{k,S} \xb_{k,S}' \right) \tilde \Db_{SS}^{-1} \Xb_S' \tilde \Wb^2 (\Wb^{\star})^{-1}  + \eb_j' (\tilde \Db_{SS})^{-1} \left( \xb_{k,S} \frac{\tilde \sigma_k^2}{\sigma^\star_k} \eb_k'\right). \\
\end{aligned}
\]
We will now focus on bounding 
\[
\begin{aligned}
&\left\| \frac{\partial \Lb_j'(\tilde \gammab)}{\partial \tilde \gammab} (\hat \gammab - \gammab^\star) \right\| \le \frac{1}{n^2} \left\| \sum_{k=1}^n \eb_j' \tilde \Db_{SS}^{-1} \left( \tilde \sigma_k^2 \xb_{k,S} \xb_{k,S}' \right) \tilde \Db_{SS}^{-1} \Xb_S' \tilde \Wb^2 (\Wb^{\star})^{-1} (\hat \gamma_k - \gamma_k^\star) \right\| \\
&+ \frac{1}{n} \left\| \sum_{k=1}^n \eb_j' (\tilde \Db_{SS})^{-1} \left( \xb_{k,S} \frac{\tilde \sigma_k^2}{\sigma^\star_k} \eb_k'\right) (\hat \gamma_k - \gamma_k^\star) \right\|.
\end{aligned}
\]
The first term is bounded by
\[
\begin{aligned}
&\frac{1}{n^2} \left\| \sum_{k=1}^n \eb_j' \tilde \Db_{SS}^{-1} \left( \tilde \sigma_k^2 \xb_{k,S} \xb_{k,S}' \right) \tilde \Db_{SS}^{-1} \Xb_S' \tilde \Wb^2 (\Wb^{\star})^{-1} (\hat \gamma_k - \gamma_k^\star) \right\| \\
&\le \frac{1}{n^2}  \| \hat \gammab - \gammab^\star \|_\infty \left\| \eb_j' \tilde \Db_{SS}^{-1} \left( \sum_{k=1}^n \tilde \sigma_k^2 \xb_{k,S} \xb_{k,S}' \right) \tilde \Db_{SS}^{-1} \Xb_S' \tilde \Wb^2 (\Wb^{\star})^{-1} \right\|\\
&= \frac{1}{n} \left\| \eb_j' \tilde \Db_{SS}^{-1} \Xb_S' \tilde \Wb^2 (\Wb^{\star})^{-1} \right\| \tilde O_\PP \left( \frac{\sqrt t}{\sqrt n} \right) = \frac{1}{n} \left\| \eb_j' \tilde \Db_{SS}^{-1} \Xb_S' \tilde \Wb^2 (\Wb^{\star})^{-1} \right\| \tilde O_\PP \left( \frac{\sqrt t}{\sqrt n} \right) \\
&= \tilde O_\PP \left( \frac{\sqrt t}{n \sqrt n} \right) \opnorm{\tilde \Db_{SS}^{-1} \Xb_S' \tilde \Wb}{} = \tilde O_\PP \left( \frac{\sqrt t}{n \sqrt n} \right) \sqrt{\Lambda_{\max}\left( n \tilde \Db_{SS}^{-1} \right) }\\
&= \tilde O_\PP \left( \frac{\sqrt t}{n} \right) \sqrt{\Lambda_{\max}\left(\tilde \Db_{SS}^{-1} \right) } = \tilde O_\PP \left( \frac{\sqrt t}{n} \right).
\end{aligned}
\]
The second term can be bounded by
\[
\begin{aligned}
&\frac{1}{n} \left\| \sum_{k=1}^n \eb_j' (\tilde \Db_{SS})^{-1} \left( \xb_{k,S} \frac{\tilde \sigma_k^2}{\sigma^\star_k} \eb_k' \right) (\hat \gamma_k - \gamma_k^\star) \right\| \le \frac 1n \left[ \max_{k \in [p]} \frac{\tilde \sigma_k}{\sigma_k^\star} \right] \left[ \max_{k \in [p]} |\hat\gamma_k - \gamma_k^\star| \right] \left\| \sum_{k=1}^n \eb_j' (\tilde \Db_{SS})^{-1} \xb_{k,S} \tilde \sigma_k  \eb_k' \right\| \\
&= \tilde O_\PP \left( \frac{\sqrt t}{n\sqrt n} \left\| \eb_j' \tilde \Db_{SS}^{-1} \Xb_S' \tilde \Wb\right\| \right) = \tilde O_\PP \left( \frac{\sqrt t}{n\sqrt n} \opnorm{\tilde \Db_{SS}^{-1} \Xb_S' \tilde \Wb}{} \right) = \tilde O_\PP \left( \frac{\sqrt t}{n} \right).
\end{aligned}
\]
We will now show that there exists a constant $q > 0$ such that 
\[
\sup_{\tilde \thetab} \opnorm{ \frac{\partial \Lb_j(\tilde \thetab)}{\partial \tilde \thetab} }{} = O_\PP(n^q)
\]
uniformly in $j$.
By the chain rule,
\[
\frac{\partial \Lb_j(\tilde \thetab)}{\partial \tilde \thetab} = \sum_{k = 1}^n \frac{\partial \Lb_j(\tilde \gammab)}{\partial \tilde \gamma_k} \frac{\partial \tilde \gamma_k}{\partial \tilde \thetab}
= \sum_{k=1}^n \frac{\partial \Lb_j(\tilde \gammab)}{\partial \tilde \gamma_k} \xb_{k,T}'.
\]
Thus,
\[
\begin{aligned}
&\opnorm{ \frac{\partial \Lb_j(\tilde \thetab)}{\partial \tilde \thetab} }{} \le \opnorm{\sum_{k=1}^n \frac{\partial \Lb_j(\tilde \gammab)}{\partial \tilde \gamma_k} \xb_{k,T}'}{} \le \sum_{k=1}^n \left\| \frac{\partial \Lb_j(\tilde \gammab)}{\partial \tilde \gamma_k} \right\| \| \xb_{k,T}' \| \le \sqrt{ \sum_{k=1}^n \left\| \frac{\partial \Lb_j(\tilde \gammab)}{\partial \tilde \gamma_k} \right\|^2 } \sqrt{\sum_{k=1}^n\ \| \xb_{k,T}' \|^2} \\
& = \sqrt{ \sum_{k=1}^n \left\| \frac{\partial \Lb_j(\tilde \gammab)}{\partial \tilde \gamma_k} \right\|^2 } \opnorm{\Xb_T}{F} = \sqrt{ \sum_{k=1}^n \left\| \frac{\partial \Lb_j(\tilde \gammab)}{\partial \tilde \gamma_k} \right\|^2 } O_\PP(\sqrt{tn}) = O_\PP(n \sqrt{t}) \max_{k \in [n]} \left\| \frac{\partial \Lb_j(\tilde \gammab)}{\partial \tilde \gamma_k} \right\|.
\end{aligned}
\]
As before,
\[
\left\| \frac{\partial \Lb_j(\tilde \gammab)}{\partial \tilde \gamma_k} \right\| \le \frac{1}{n^2} \left\| \eb_j' \tilde \Db_{SS}^{-1} \left( \tilde \sigma_k^2 \xb_{k,S} \xb_{k,S}' \right) \tilde \Db_{SS}^{-1} \Xb_S' \tilde \Wb^2 (\Wb^{\star})^{-1} \right\| + \frac{1}{n} \left\| \eb_j' (\tilde \Db_{SS})^{-1} \xb_{k,S} \frac{\tilde \sigma_k^2}{\sigma^\star_k} \eb_k' \right\|
\]
Controlling the first term,
\[
\begin{aligned}
&\left\| \eb_j' \tilde \Db_{SS}^{-1} \left( \tilde \sigma_k^2 \xb_{k,S} \xb_{k,S}' \right) \tilde \Db_{SS}^{-1} \Xb_S' \tilde \Wb^2 (\Wb^{\star})^{-1} \right\| \le \left\| \tilde \Db_{SS}^{-1} \tilde \sigma_k \xb_{k,S} \right\|^2 \opnorm{ \Xb_S' \tilde \Wb^2 (\Wb^{\star})^{-1}}{}\\
& \opnorm{ \Xb_S' \tilde \Wb^2 (\Wb^{\star})^{-1}}{} = \opnorm{ \Xb_S' \tilde \Wb }{} O_\PP(1) = \sqrt{n \Lambda_{\max}( \tilde \Db_{SS} )} O_\PP(1) = \tilde O_\PP(\sqrt n)\\
& \left\| \tilde \Db_{SS}^{-1} \tilde \sigma_k \xb_{k,S} \right\|^2 \le \sum_{k = 1}^n \left\| \tilde \Db_{SS}^{-1} \tilde \sigma_k \xb_{k,S} \right\|^2 = \sum_{k = 1}^n \tr \left( \tilde \Db_{SS}^{-1} \tilde \sigma^2_k \xb_{k,S} \xb_{k,S}' \tilde \Db_{SS}^{-1} \right) = n \tr \left( \tilde \Db_{SS}^{-1} \right) \\
&\le n^2 \Lambda_{\max} ( \tilde \Db_{SS}^{-1}) = \tilde O_\PP(n^2).
\end{aligned}
\]
Similarly controlling the second term,
\[
\left\| \eb_j' (\tilde \Db_{SS})^{-1} \xb_{k,S} \frac{\tilde \sigma_k^2}{\sigma^\star_k} \eb_k' \right\| \le \left\|(\tilde \Db_{SS})^{-1} \xb_{k,S} \frac{\tilde \sigma_k^2}{\sigma^\star_k} \right\| = \left\| (\tilde \Db_{SS})^{-1} \xb_{k,S} \tilde \sigma_k \right\| O_\PP(1) = \tilde O_\PP(n).
\]
Notice that all of the above bounds are uniform in $j$ and $\tilde \gammab \in \Gamma$.  Combining these we obtain,
\[
\left\| \frac{\partial \Lb_j(\tilde \gammab)}{\partial \tilde \gamma_k} \right\| = \tilde O_\PP (\sqrt n)
\]
and so uniformly,
\[
\sup_{\tilde \thetab} \opnorm{ \frac{\partial \Lb_j(\tilde \thetab)}{\partial \tilde \thetab} }{} = \tilde O_\PP(n \sqrt t) = o_\PP(n^{5/4}).
\]
\end{proof}

We can use Lemma \ref{lem:beta_theta_grad} to control the Lipschitz constant of $\hat \betab$ by 
\[
\left\| \frac{\partial \Lb_j'(\tilde \thetab)}{\partial \tilde \thetab} \epsilon \right\| \le \opnorm{ \frac{\partial \Lb_j'(\tilde \thetab)}{\partial \tilde \thetab}}{} \| \epsilon \| = O_\PP \left( n^{q+1/2} \right)
\]
Consider the covering of the space $\Theta$ (which lies within the unit ball for large enough $n$) from Lemma \ref{lem:lipschitz_cover} with $\xi = 1 / \sqrt n$ and let $\Kcal$ be cluster centers that lie within $\Theta$ (these may be just arbitrary choices of elements from the clusters).
With probability $1 - \delta$, for all $j \in [p]$ and $\hat \thetab \in \Kcal$, 
\[
\begin{aligned}
&|\hat \betab_j (\hat \thetab) - \hat \betab_j(\thetab^\star)| = |(\Lb_j(\hat \thetab) - \Lb_j(\thetab^\star))' \epsilonb| \le \| \Lb_j(\hat \thetab) - \Lb_j(\thetab^\star)\| \sqrt{2 \log (p |\Kcal| / \delta)} \\
&\le \sup_{\tilde \thetab} \left\| \frac{\partial \Lb'_j(\tilde \gammab)}{\partial \tilde \gammab} (\hat \gammab - \gammab^\star) \right\| \sqrt{2 \log (p |\Kcal| / \delta)} = \tilde O_\PP\left(\frac{t}{n} \right) = \tilde o_\PP\left(\frac{1}{\sqrt n} \right)
\end{aligned}
\]
Hence, by Lemma \ref{lem:lipschitz_cover} and the triangle inequality, 
\[
\forall \hat \thetab \in \Theta, \quad \|\hat \betab_S(\hat \thetab) - \hat \betab_S(\thetab^\star)\|_\infty = \tilde o_\PP\left(\frac{1}{\sqrt n} \right)
\]
uniformly.
And as a result of stage 2, $\hat \thetab \in \Theta$ for large enough $C,n$ by Theorem \ref{thm:unknown_mean}.
Furthermore,
\[
\|\hat \betab_S(\hat \thetab) - \hat \betab_S(\thetab^\star)\|_2 \le \sqrt s \|\hat \betab_S(\hat \thetab) - \hat \betab_S(\thetab^\star)\|_\infty = \tilde o_\PP\left(\frac{\sqrt s}{\sqrt n} \right).
\]
\end{proof}

Consider the error for the optimal WLS estimate,
\[
\hat \betab_S(\thetab^\star) - \betab_S^\star = (\Xb_S' (\Wb^\star)^2 \Xb_S)^{-1} \Xb_S' \Wb^\star \epsilonb = \frac 1n (\Db^\star_{SS})^{-1} \Xb_S' \Wb^\star \epsilonb.
\]
Hence, $\hat \betab_S(\thetab^\star) - \betab_S^\star$ is a zero mean Gaussian with variance $\frac 1n (\Db^\star_{SS})^{-1}$.
Therefore, 
\begin{equation}
  \label{eq:proof:scad_mean:1}
  \norm{\hat \betab_S(\thetab^\star) - \betab_S^\star}_\infty = O_\PP\left( \sqrt{\max_{j \in S} (\Db^\star_{SS})^{-1}_{jj} \frac{\log s}{n}} \right)
\end{equation}
Under the assumption that $(\Db_{SS}^{-1})_{jj} = O_\PP(1)$, we have that
\begin{equation*}
\begin{aligned}
  \min_{j \in S} |\hat \beta_j(\hat \thetab)| &
    \geq \min_{j \in S} |\beta^\star_j| - \norm{\hat\betab_S(\hat \thetab) - \hat \betab_S(\thetab^\star)}_\infty - \norm{\hat \betab_S(\thetab^\star) - \betab_S^\star}_\infty\\
    & \geq \beta_{\min} - O_\PP\left( \frac{\sqrt{\log s}}{\sqrt n} \right) = \omega_\PP(\lambda \| \hat \Wb \Xb_j\|)\\
\end{aligned}
\end{equation*}
as long as $\lambda = o(\beta_{\min}/\sqrt{n}) $ and $\| \hat \Wb \Xb_j \| = O_\PP(\sqrt n)$ (which holds as long as $\|\Wb^\star \Xb_j \| = O_\PP(\sqrt n)$).
Furthermore,
\[
\Xb\uSc' \hat \Wb (\hat \Wb \yb - \hat \Wb \Xb\hat\betab) = \Xb\uSc'\hat\Wb (\Ib - \hat \Pcal) \hat \Wb \yb = \Xb\uSc' \hat \Wb (\Ib - \hat \Pcal) \hat \Wb {\Wb^\star}^{-1} \epsilonb.
\]
Hence,
\[
| \Xb_j' \hat \Wb (\hat \Wb \yb - \hat \Wb \Xb\hat\betab) | = | \Xb_j' \hat \Wb (\Ib - \hat \Pcal) \hat \Wb {\Wb^\star}^{-1} \epsilonb | \le | \Xb_j' \hat \Wb^2 {\Wb^\star}^{-1} \epsilonb |.
\]
Because $\| \hat \sigmab / \sigmab^\star - \one \|_\infty = o_\PP(1)$,
\[
|\Xb_j' \hat \Wb^2 \Wb^\star \epsilonb| \le |\Xb_j' \hat \Wb \epsilonb| + |\Xb_j' \hat \Wb (\hat \Wb (\Wb^\star)^{-1} - I) \epsilonb| \le |\Xb_j' \hat \Wb \epsilonb| (1 + o_\PP(1))
\]
which is uniform in $j$. So,
\[
|\Xb_j' \hat \Wb \epsilonb| = |\Xb_j' \Wb^\star \epsilonb| (1 + o_\PP(1)) = O_\PP (\| \Wb^\star \Xb_j \| \sqrt{\log p}) = O_\PP(\| \hat \Wb \Xb_j \| \sqrt{\log p})
\]
uniformly over $j$.
Hence, if $\lambda = \omega(\sqrt{\log p} / n)$ then
\[
| \Xb_j' \hat \Wb (\hat \Wb \yb - \hat \Wb \Xb\hat\betab) | = O_\PP ( n \lambda \| \hat \Wb \Xb_j \| ).
\]
uniformly over $j \in [p]$.
In summary, the zero subgradient conditions hold as long as 
\[
\beta_{\min} = \omega\left( \frac{\sqrt{\log p}}{\sqrt n} \right), \quad \lambda = o\left(\frac{\beta_{\min}}{\sqrt n} \right), \quad \lambda = \omega \left( \frac{\sqrt{\log p}}{n} \right).
\]

\end{document}